\documentclass[article]{elsarticle}
\usepackage{hyperref}
\usepackage{amsmath,amsxtra,amssymb,latexsym,epsfig,amscd,amsthm,fancybox,epsfig}
\usepackage[mathscr]{eucal}
\usepackage{graphicx}
\usepackage{epsfig} 
\usepackage{epstopdf}
\usepackage{cases}
\usepackage{subfig}
\usepackage{color}

\newtheorem{thm}{Theorem}[section]
\newtheorem {asp}{Assumption}[section]
\newtheorem{lm}{Lemma}[section]

\newtheorem{deff}{Definition}[section]

\newtheorem{prop}{Proposition}[section]
\theoremstyle{definition}

\theoremstyle{remark}
\newtheorem{rem}{Remark}

\numberwithin{equation}{section}

\newcommand{\eps}{\varepsilon}

\newcommand{\A}{\mathcal{A}}

\newcommand{\C}{\mathcal{C}}

\newcommand{\M}{\mathcal{M}}
\newcommand{\F}{\mathcal{F}}

\newcommand{\E}{\mathbb{E}}

\newcommand{\N}{\mathbb{N}}

\newcommand{\PP}{\mathbb{P}}

\newcommand{\K}{\mathcal{K}}
\newcommand{\U}{\mathcal{U}}

\newcommand{\R}{\mathbb{R}}

\newcommand{\Lom}{\mathcal{L}}

\newcommand{\abs}[1]{\left\vert#1\right\vert}

\newcommand{\norm}[1]{\left\Vert#1\right\Vert}

\renewcommand{\vec}[1]{\mathbf{#1}}

\newcommand{\bdelta}{\boldsymbol{\delta}}
\newcommand{\Vphi}{\boldsymbol\varphi}
\newcommand{\Vrho}{\boldsymbol\rho}
\newcommand{\Bphi}{\boldsymbol\phi}
\numberwithin{equation}{section}

\newcommand{\bed}{\begin{displaymath}}
\newcommand{\eed}{\end{displaymath}}
\newcommand{\bea}{\bed\begin{array}{rl}}
\newcommand{\eea}{\end{array}\eed}

\newcommand{\barray}{\begin{array}{ll}}
\newcommand{\earray}{\end{array}}
\newcommand{\diag}{{\rm diag}}
\newcommand{\conv}{{\rm Conv}}

\newcommand{\1}{\boldsymbol{1}}

\def\bar{\overline}
\def\hat{\widehat}
\def\wdt{\widetilde}

\def\a.s{\text{\;a.s.\;}}

\def\bdelta{\boldsymbol{\delta}}
\def\Ephi{\E_{\Bphi}}
\def\PPphi{\PP_{\Bphi}}
\newcommand{\ko}{Kolmogorov }

\newcommand{\DD}{\mathbb{D}}

\begin{document}
\begin{frontmatter}
	
	\title{Stochastic Functional Kolmogorov Equations II: Extinction\tnoteref{mytitlenote}}
	\tnotetext[mytitlenote]{The research of D. H. Nguyen was supported in part by the National Science Foundation under grant under grant DMS-1853467. 
		The research of N. N. Nguyen and G. Yin was supported in part by the National Science Foundation under grant under grant DMS-2114649.}
	\author[myaddress1]{Dang H. Nguyen}
	\ead{dangnh.maths@gmail.com.}
	\author[myaddress]{Nhu N. Nguyen}
	\ead{nguyen.nhu@uconn.edu}
	\author[myaddress]{George Yin\corref{mycorrespondingauthor}}
	\cortext[mycorrespondingauthor]{Corresponding author}
	\ead{gyin@uconn.edu}
	
	\address[myaddress1]{Department of Mathematics, University of Alabama, Tuscaloosa, AL
		35487, USA}
	\address[myaddress]{Department of Mathematics, University of Connecticut, Storrs, CT
		06269, USA}
\begin{abstract}
This work, Part II, together with its companion Part I
develops a new framework
for stochastic functional Kolmogorov equations, which are nonlinear stochastic differential equations depending on the current as well as the past states. Because of the complexity of the problems, it is natural to divide our contributions into two parts to answer a long-standing question in biology and ecology. What are the minimal conditions for long-term persistence and extinction of a population?
Part I of our work provides characterization of persistence, whereas in this part, extinction is the main focus.
The main techniques used in this paper are combination of the newly developed functional It\^o formula and a
dynamical system approach.
Compared to the study of stochastic Kolmogorov systems without delays, the main difficulty is that infinite dimensional systems have to be treated. The extinction is characterized after investigating random occupation measures and examining behavior of functional systems around boundaries. Our characterizations of long-term behavior of the systems reduces to that of Kolmogorov systems without delay when there is no past dependence. A number of applications are also examined.



 \bigskip
 \noindent {\bf Running Title.} Stochastic Functional Kolmogorov Equations

\end{abstract}
\end{frontmatter}
\tableofcontents

\section{Introduction}
Motivated by a wide variety of applications in ecology and biology, we aim to develop a new framework
 of stochastic functional \ko equations.
To keep our work with a manageable length, we divide the contributions to two parts.
 We aim to answer the long-standing
question of fundamental importance pertaining to biology and ecology: What are the minimal (necessary and sufficient) conditions for long-term persistence and extinction of a population?
This work, Part II, focuses on extinction, whereas its companion Part I concentrates on persistence.
  The extinction and persistence are phenomena
go far beyond biological and ecological systems.
  Such long-term properties are shared by all processes of \ko type.
 One of the main difficulties is that we need to deal with
  infinite dimensional processes.

Taking random fluctuations of the environment into consideration,
 a stochastic \ko differential equation is given by
\begin{equation}\label{kol-2d-ran}
dx_i(t)=x_i(t) f_i(x_1(t),\dots,x_n(t))dt+x_i(t)g_i(x_1(t),\dots x_n(t))dB_i(t),\quad i=1,\dots,n,
\end{equation}
where $B_1(t),\dots,B_n(t)$ are $n$ real-valued Brownian motions
(independent or not).
Such
stochastic \ko system are used extensively in the modeling and analysis of ecological and biological systems such as Lotka-Volterra
predator-prey models, Lotka-Volterra competitive model, replicator dynamic systems, stochastic epidemic models, stochastic tumor-immune system, and stochastic chemostat models, among others.
Apart from ecological and biological systems,
 numerous problems arising in mathematical physics, statistical mechanics, and many related fields, use \ko nonlinear stochastic differential equations. For example, the
  generalized Ginzburg-Landau equation
 is used for
bistable systems, chemical turbulence, phase transitions in non-equilibrium systems,
 among other.
Because of its prevalence in applications, \ko equations \eqref{kol-2d-ran} have  attracted much attention in the past decades;
substantial progress has been made.
To proceed, let us briefly recall some important works of the developments to date.
Some
early mathematical formulations
were introduced by
Verhulst \cite{Ver38} for logistic models,
by Lotka and Volterra \cite{Lot25,Vol26} for Lotka-Volterra systems, and
by Kermack and McKendrick \cite{Ker27,Ker32} for infectious diseases.
 By now, \ko
  stochastic population systems (using stochastic differential equations or difference equations)
 together with their longtime behavior have been relatively well understood; see
\cite{BL16, RS14, SBA11} for  \ko stochastic systems in compact domains and
\cite{Ben14, HN18} for  \ko systems in non-compact domains,
\cite{Die16, DN18,Lou,NNY19,DYZ20,TNY20} for variants of \ko systems such as epidemic models, migration and spatial heterogeneity on single and multiple species,
chemostat models,  tumor-immune system,
and \cite{Imh05,Imh09} for replicator dynamics.
For the most recent
 developments and substantial progress, we refer to
Bena\"im \cite{Ben14},
Henning and Nguyen
\cite{HN18}, Schreiber and Bena\" im \cite{SBA11}, and references therein.

In contrast to existing works, our work in this paper and its companion \cite{NNY21} aim to
substantially advance the existing literature for
  a class of $n$-dimensional stochastic functional \ko systems allowing
 delay and past dependence, so as to provide essential utility to a wide range of applications.
 Clearly, the delays or past dependence are unavoidable natural phenomena in dynamical systems;  the framework of
 stochastic functional differential equations
 is more realistic, more effective, and more general
for the population dynamics in real life
than a stochastic differential equation counterpart.
In population dynamics, some  delay mechanisms studied in the literature include age structure, feeding times, replenishment or regeneration time for resources  \cite{JC13}.
Although there are many excellent treatises of \ko stochastic differential equations, the work on \ko stochastic differential equations with delay is relatively scarce.
In addition, other than the specific models and applications treated, there has not been a unified framework
and a systematic treatment for \ko stochastic functional differential systems yet, and
most of the existing results involving delay are not as sharp as desired.
New methods and techniques need to be developed to carry out the analysis.
There is a strong motivation and pressing need
to develop a unified framework for stochastic functional \ko systems.

 While the models with functional stochastic differential equations are more realistic and more general, the analysis of such systems is far more difficult.
Perhaps, part of the difficulties in studying stochastic delay systems is that there had been virtually no bona fide
operators and functional
It\^o formulas except some general setup in a Banach space such as \cite{SE86} before 2009.
In \cite{Dup09}, Dupire generalized the It\^o formula to a functional setting by using  pathwise functional derivatives.
The It\^o formula  developed has substantially eased the difficulties and encouraged subsequent development with numerous applications.
His work was developed further by Cont and Fourni\'e \cite{CF10, CF13}.
Using the newly developed functional It\^o formula
enables us to analyze effectively the segment processes in
the stochastic functional \ko equations.
Moreover, while the solutions of stochastic differential equations (without delays) are Markovian processes, the solutions of stochastic differential equations with delay is non-Markov. One  uses the so-called segment processes for the delay equations. However, such segment processes live in an infinite dimensional space. Many of the known results in the usual stochastic differential equation setup  are no longer applicable.  Because \ko is highly nonlinear, analyzing such systems with delay becomes even more difficult.

Our goal in this paper,
 is to characterize  the long-term behavior
focusing on extinction.
The results of this paper combined with the persistence presented in the first part \cite{NNY21} provides a complete long-term characterization for the stochastic functional \ko system.
We show that our results will cover, improve, and outperform
 many existing results of \ko systems (with and without delays) such
as the study on stochastic delay Lotka-Volterra competitive models
\cite{Mao04,Liu17}, the work on
stochastic delay Lotka-Volterra predator-prey models \cite{Gen17,Liu13,Liu17-dcds,Liu17-1,Wu19}, the treatment of
stochastic delay epidemic SIR models \cite{Che09,QLiu16-1,QLiu16,QLiu16-2,Mah17},
and the study on
stochastic delay chemostat models \cite{Sun18,Sun18-1,Zha19}
and the delayed replicator models.
It should be mentioned that
 for replicator dynamics,  it seems to be no investigation of stochastic delay systems to date to the best of our knowledge.
It is also worth noting that our sufficient conditions for extinction in this part and the conditions of persistence in \cite{NNY21}) are almost necessary in the sense that excluding critical cases, if the system is not persistent, the extinction will happen and vice versa.

By combining the newly developed functional It\^o formula and the dynamical system point of view,
we advance knowledge to treat infinite dimensional systems.
Characterization of the extinction is obtained after investigating the random occupation measures and examining carefully the behavior of a functional system around the boundary.
It is worth noting that
handling random occupation measures in an infinite dimensional space to obtain the tightness and its limit is far more challenging.
Examining the
behavior of solutions near the boundary for functional \ko systems requires more delicate analysis than
that for systems without delay
because
even if
the state is close to the boundary, its history may not be.

The rest of the paper is organized as follows.
Section \ref{sec:res} presents the formulation of our problem and
main results of
stochastic functional \ko systems.
Section \ref{sec:key} recalls basic properties of \ko equations with delays, including
 well-posedness of the system,
 positivity of solutions,
 which have been proved in  \cite{NNY21}.
We then study
the tightness of families of random occupation measures and the convergence properties.
Next, Section \ref{sec:ext} studies the conditions for extinction
 of \ko systems.
Finally, Section \ref{sec:app} provides
several
applications involving \ko dynamical systems and detailed accounts on how to use our results on stochastic functional \ko equations to treat each of the application examples.

\section{Formulation and Main Results}\label{sec:res}
We first provide a glossary of
 symbols and notation to be used
in this paper to help the reading.

\begin{tabbing}
\hspace*{0.7in}\mbox{ }\=  \kill
$r$\> a fixed positive number\\
$\abs{\cdot}$\>  Euclidean norm\\
$\C([a; b];\R^n)$\> set of $\R^n$-valued continuous
functions defined on $[a; b]$\\
$\C$\> $:= \C([-r; 0];\R^n)$\\
$\Vphi$\> $=(\varphi_1,\dots,\varphi_n)\in \C$\\
$\vec x$\>$=(x_1,\dots,x_n):=\Vphi(0)\in\R^n$\\
$\norm{\Vphi}$\> $:=
\sup\{\abs{\Vphi(t)}: t\in [-r,0]\}$\\
$\vec X_t$\> $:=\vec X_t(s):= \{\vec X(t + s) : -r\leq s\leq 0\}$ (segment function)\\
$X_{i,t}$\> $:=X_{i,t}(s):= \{X_i(t + s) : -r\leq s\leq 0\}$\\
$\C_+$\> $:=\left\{\Vphi=(\varphi_1,\dots,\varphi_n)\in\C: \varphi_i(s)\geq 0\;\forall s\in[-r,0],i=1,\dots,n\right\}$\\
$\partial\C_+$\> $:=\left\{\Vphi=(\varphi_1,\dots,\varphi_n)\in\C: \norm{\varphi_i}=0\text{ for some }i=1,\dots,n\right\}$\\
$\C^{\circ}_+$\> $:=\left\{\Vphi\in\C_+: \varphi_i(s)> 0,\forall s\in[-r,0], i=1,\dots,n\right\}\neq \C_+\setminus\partial\C_+$\\
$\norm{\Vphi}_\alpha$\> $:=\norm{\Vphi}+\sup_{-r\leq s< t\leq 0}\frac{\abs{\Vphi(t)-\Vphi(s)}}{(t-s)^\alpha},$ for some $0<\alpha<1$\\
$\C^\alpha$\>  space of  H\"older continuous functions
endowed with the norm $\norm{\cdot}_\alpha$\\
$\Gamma$\> $n\times n$ matrix\\
$\Gamma^\top$\> transpose of $\Gamma$\\
$\vec B(t)$\> $=(B_1(t),\dots, B_n(t))^\top$, a $n$-dimensional standard Brownian motion \\
$\vec E(t)$\> $=(E_1(t),\dots, E_n(t))^\top:=\Gamma^\top\vec B(t)$\\
$\Sigma$\> $=(\sigma_{ij})_{n\times n}:= \Gamma^\top \Gamma$\\
$\M$\>set of ergodic invariant probability measures of $\vec X_t$ supported on $\partial \C_+$\\
$\conv(\M)$\>convex hull of $\M$\\
$\vec 0$\> the zero constant function in $\C$\\
$\bdelta^*$\>the Dirac measure concentrated at $\vec 0$\\
$\1_A$\> the indicator function of set $A$\\
$D_{\eps,R}$\>$:=\left\{\Vphi\in\C_+:\norm{\Vphi}\leq R, x_i\geq \eps\;\forall i;\vec x:=\Vphi(0)\right\}, \eps, R>0$\\
$\DD$\> space of Cadlag functions mapping $[-r,0]$ to $\R^n$\\
$A_0,A_1,A_2$\>constants satisfying Assumption \ref{asp1}\\
$\gamma_0,\gamma_b,M$\>constants satisfying Assumption \ref{asp1}\\
$\vec c,h(\cdot),\mu$\>vector, function and probability measure satisfying Assumption \ref{asp1}\\
$\wdt K,b_1,b_2$\> constants satisfying Assumption \ref{asp-2}\\
$h_1(\cdot),\mu_1$\>function and probability measure satisfying Assumption \ref{asp-2}\\
$B_0,B_1,B_2$\> constants satisfying Assumption \ref{a.extn2}\\
$B_3,p_2$\> constants satisfying Assumption \ref{a.extn2}\\
$I$\> a subset of $\{1,\dots,n\}$\\
$I^c$\>:=$\{1,\dots,n\}\setminus I$\\
$
\C_+^{I}$\>$:=\left\{\Vphi\in\C_+: \norm{\varphi_i}=0\text{ if }i\in I^c \right\}$\\
$
\C_+^{I,\circ}$\>$:=\left\{\Vphi\in\C_+: \norm{\varphi_i}=0\text{ if }i\in I^c \text{ and } \varphi_i(s)>0\;\forall s\in [-r,0]\text{ if }i\in I\right\}
$\\
$\partial \C_+^I$\>$:=\left\{\Vphi\in\C_+: \norm{\varphi_i}=0\text{ if }i\in I^c\text{ and }\norm{\varphi_i}=0\text{ for some }i\in I\right\}$\\
$\M^I$\> sets of ergodic invariant probability measure on $\C^I_+$\\
$\M^{I,\circ}$\> sets of ergodic invariant probability measure on $\C^{I,\circ}_+$\\
$\partial\M^{I}$\> sets of ergodic invariant probability measure on $\partial\C^I_+$\\
$I_\pi$\> the subset of $\{1,\dots,n\}$ such that $\pi(\C_+^{I_\pi,\circ})=1, \pi\in\M$\\
$\gamma,p_0,A$\>constants satisfying the condition in Lemma \ref{LV}\\
$\Vrho$\> $=(\rho_1,\dots,\rho_n)$ vector satisfying the condition in Lemma \ref{LV}\\
$ V_{\Vrho}(\Vphi)$\>$:=\Big(1+\vec c^\top\vec x\Big)\prod_{i=1}^nx_i^{\rho_i}\exp\left\{A_2\int_{-r}^0\mu(ds)\int_s^0 e^{\gamma(u-s)}h\big(\Vphi(u)\big)du\right\}$\\[0.5ex]
$ V_{\vec 0}(\Vphi)$\>$:=\Big(1+\vec c^\top\vec x\Big)\exp\left\{A_2\int_{-r}^0\mu(ds)\int_s^0 e^{\gamma(u-s)}h\big(\Vphi(u)\big)du\right\}$\\[0.5ex]
$\C_{V, M}$\>$:=\{\Vphi\in\C_+: A_2\gamma \int_{-r}^0\mu(ds)\int_s^0e^{\gamma(u-s)}h\big(\Vphi(u)\big)du\leq A_0 \text{ and } |\vec x|\leq M\}$\\
$n^*$\> constant satisfying $\gamma_0 (n^*-1)- A_0>0$\\
$p_1$\> constant satisfying condition \eqref{cd-p0} and $p_1>p_0$\\
$\C_V(\hat\delta)$\>$:=\{\Vphi\in\C^{\circ}_+\cap\C_{V,M} \text{ and }|\varphi_i(0)|\leq\hat\delta \text{ for some }i\}$\\
$\hat \alpha_i,\alpha_*,\kappa_e$\>constants satisfying \eqref{e3.3}\\
$T_e,\delta_e$\>constants determined in Lemma \ref{lm4.2}
\end{tabbing}

In this paper, we
treat a stochastic delay Kolmogorov system of equations
\begin{equation}\label{Kol-Eq}
\begin{cases}
dX_i(t)=X_i(t)f_i(\vec X_t)dt+X_i(t)g_i(\vec X_t)dE_i(t), \quad i=1,\dots,n,\\
\vec X_0=\Bphi\in \C_+.\end{cases}
\end{equation}
Denote the solution of \eqref{Kol-Eq} by $\vec X^{\Bphi}(t)$.
 For convenience, we  usually suppress the superscript ``$\Bphi$'' and
use $\PPphi$ and $\Ephi$ to denote the probability and expectation given the initial value $\Bphi$, respectively. We also assume that the initial value is non-random. Denote by $\{\F_t\}_{t\geq 0}$
the filtration satisfying the usual conditions and assume that the $n$-dimensional Brownian motion $\vec B(t)$ is adapted to $\{\F_t\}_{t\geq 0}$.
Note that a segment process is also referred to as a memory
segment function.
We assume the following assumptions hold in the rest of the paper.

\begin{asp}\label{asp1}
{\rm
The coefficients of \eqref{Kol-Eq} satisfy
\begin{itemize}
\item [\rm(1)] $\diag(g_1(\Vphi),\dots,g_n(\Vphi))\Gamma^\top\Gamma\diag(g_1(\Vphi),\dots,g_n(\Vphi))=(g_i(\Vphi)g_j(\Vphi)\sigma_{ij})_{n\times n}$ is a positive definite matrix for any $\Vphi\in\C_+$.
\item [\rm(2)] $f_i(\cdot),g_i(\cdot): \C_+\to\R$ are Lipschitz continuous in each bounded set of $\C_+$ for any $i=1,\dots,n.$
\item [\rm(3)] There exist $\vec c=(c_1,\dots,c_n)\in \R^{n}$, $c_i>0,\forall i$ and $\gamma_b,\gamma_0>0$, $A_0>0$, $A_1>A_2>0$, $M>0$, a continuous function $h:\R^n\to\R_+$ and a probability measure $\mu$ concentrated on $[-r,0]$ such that for any $\Vphi\in\C_+$
\begin{equation}\label{asp1-3}
\begin{aligned}
\dfrac{\sum_{i=1}^n c_ix_if_i(\Vphi)}{1+\vec c^\top \vec x}&-\dfrac12\dfrac{\sum_{i,j=1}^n\sigma_{ij}c_ic_jx_ix_jg_i(\Vphi)g_j(\Vphi)}{(1+\vec c^\top \vec x)^2}+\gamma_b\sum_{i=1}^n\Big(\abs{f_i(\Vphi)}+g_i^2(\Vphi)\Big)
\\&\leq A_0 \1_{\left\{\abs{\vec x}< M\right\}}-\gamma_0-A_1h(\vec x)+A_2\int_{-r}^0h\big(\Vphi(s)\big)\mu(ds),
\end{aligned}
\end{equation}
where $\vec x:=\Vphi(0).$ We assume without loss of generality that $h:\R^n\to[1,\infty)$,
otherwise, we can always change $\gamma_0$ and $A_1, A_2$ to fulfill this requirement.
\end{itemize}
}
\end{asp}

\begin{asp}\label{asp-2}
{\rm
One of following conditions holds:
\begin{itemize}
\item [\rm(a)] There is a constant $\wdt K$ such that for any $\Vphi\in \C_+$, $\vec x=\Vphi(0)$
\begin{equation}\label{A1}
\sum_{i=1}^n\abs{f_i(\Vphi)}+\sum_{i=1}^n g_i^2(\Vphi)\leq \wdt K \Big[h(\vec x)+\int_{-r}^0 h(\Vphi(s))\mu(ds)\Big].
\end{equation}
\item [\rm(b)] There exist $b_1,b_2>0$,
a function $h_1:\R^n\to[1,\infty]$, and
a probability measure $\mu_1$ on $[-r,0]$ such that for any $\Vphi\in \C_+$, $\vec x=\Vphi(0)$
\begin{equation}\label{A2}
b_1h_1(\vec x)\leq \sum_{i=1}^n\abs{f_i(\Vphi)}+\sum_{i=1}^n g_i^2(\Vphi)\leq b_2\Big[h_1(\vec x)+\int_{-r}^0h_1(\Vphi(s))\mu_1(ds)\Big].
\end{equation}
\end{itemize}
}
\end{asp}

\begin{rem}

\begin{itemize}
\item
The assumptions above
and additional assumptions to follow
are not restrictive
and are easily verifiable.
Such conditions are common in
population dynamics
 in the literature; see Section \ref{sec:app}.

\item
Parts (2) and (3) of Assumption \ref{asp1} guarantee the existence and uniqueness of a strong solution to \eqref{Kol-Eq}.
We need part (1) of Assumption \ref{asp1} to ensure that the solution to \eqref{Kol-Eq} is a non-degenerate diffusion.
Moreover,  as will be shown later
that (3) implies the tightness of the family of transition probabilities associated with
 the solution to \eqref{Kol-Eq}.
 One difficulty stems from the positive term $A_2\int_{-r}^0h\big(\Vphi(s)\big)\mu(ds)$ on the right-hand side of \eqref{asp1-3}, which cannot be relaxed in practice.
\item
Assumption \ref{asp-2} plays an important role in ensuring the $\pi$-uniform integrability of $\sum_i \big(\abs{f_i(\cdot)}+g_i^2(\cdot)\big)$, for any invariant measure $\pi$.
\end{itemize}
\end{rem}

As was mentioned, persistence and extinction are concepts of vital importance in biology and ecology. It turns out that such concepts are features in all stochastic functional \ko systems.
The termination of a species in biology is referred to as extinction, the moment of extinction is generally considered to be the death of the last individual of the species.
In contrast
 to extinction, we have the persistence of a species.
We first define persistence and extinction
similar to \cite{HN18, SS12, SBA11}.

\begin{deff} \label{def-persist} {\rm
Let $\vec X(t)= (X_1(t),\dots,X_n(t))^\top$ be the solution of \eqref{Kol-Eq}.
	The process $\vec X$ is strongly stochastically persistent if for any $\eps>0$, there exists an $R>0$ such that for any $\Bphi\in \C_+^\circ$
	\begin{equation}
	\liminf_{t\to\infty}\PPphi\left\{R^{-1}\leq \abs{X_i(t)}\leq R\right\}\geq 1-\eps\text{ for all }i=1,\dots,n.
	\end{equation}
}\end{deff}

\begin{deff} {\rm With $\vec X(t)$ given in Definition \ref{def-persist}, for  $\Bphi\in\C_+^\circ $
and some  $i \in \{1,\dots,n\}$, we say
$X_i$ goes extinct with probability $p_{\Bphi}>0$ if
	\[
	\PPphi\left\{\lim_{t\to\infty}X_i(t)=0\right\}=p_{\Bphi}.
	\] }
\end{deff}

Let $\M$ be the set of ergodic invariant probability measures of $\vec X_t$ supported on the boundary $\partial \C_+$.  Letting $\bdelta^*$ be the Dirac measure concentrated at $\vec 0$, then $\bdelta^*\in\M$ so that $\M\neq\emptyset$. For a subset $\wdt\M\subset \M$, denote by $\conv(\wdt\M)$ the convex hull of $\wdt\M$ (the set of probability measures $\pi$ of the form
$\pi(\cdot)=\sum_{\nu\in\wdt\M}p_\nu\nu(\cdot)$
with $p_\nu\geq 0$  and $\sum_{\nu\in\wdt\M}p_\nu=1$).
{\color{blue} 
	For any $\pi\in\conv(\M)$, we define
\begin{equation*}
\lambda_i(\pi):=\int_{\partial \C_+}\left(f_i(\Vphi)-
\frac{\sigma_{ii}g_i^2(\Vphi)}{2}\right)\pi(d\Vphi).
\end{equation*}
}

For a subset $I$ of $\{1,\dots,n\}$, denote
$I^c:=\{1,\dots,n\}\setminus I$,
$$
\C_+^{I}:=\left\{\Vphi\in\C_+: \norm{\varphi_i}=0\text{ if }i\in I^c \right\},
$$
$$
\C_+^{I,\circ}:=\left\{\Vphi\in\C_+: \norm{\varphi_i}=0\text{ if }i\in I^c \text{ and } \varphi_i(s)>0\text{ for all }s\in [-r,0]\text{ if }i\in I\right\},
$$
and
$$\partial \C_+^I:=\left\{\Vphi=(\varphi_1,\dots,\varphi_n)\in\C: \norm{\varphi_i}=0\text{ if }i\in I^c\text{ and }\norm{\varphi_i}=0\text{ for some }i\in I\right\}.$$
In case $I=\emptyset$,
$\C_+^I=\C_+^{I,\circ}=\{\vec 0\}.$
Denote by
$\M^I, \M^{I,\circ}, \partial M^I$  the sets of ergodic measures on $\C_+^I,\C_+^{I,\circ}$
and $\partial C_+^I$, respectively.

Consider $\pi\in\M\setminus\{\bdelta^*\}$.
Since the diffusion $\vec X_t$ is non-degenerate in each subspace,
there exists a subset of $\{1,\dots,n\}$, denoted by $I_\pi$ such that
$\pi(\C_+^{I_\pi,\circ})=1$.
The following conditions will imply persistence cannot happen.

\begin{asp}\label{a.extn}{\rm
There exists a subset $I\subset\{1,\dots,n\}$ such that
\begin{equation}\label{ae3.1}
\max_{i\in I_\pi^c,\; \pi\in\M^{I,\circ}}\{\lambda_i(\pi)\}<0.
\end{equation}
If $I\neq \emptyset$, we assume further that
\begin{equation}\label{ae3.2}
\max_{i\in I}\{\lambda_i(\nu)\}>0,
\end{equation}
for any $\nu\in \conv(\partial \M^I)$.
}\end{asp}

\begin{asp}\label{asp-ginverse}{\rm
The inverse of the matrix $(x_ix_j\sigma_{ij}g_i(\Vphi)g_j(\Vphi))_{n\times n}$ is uniformly bounded in $D_{\eps,R}$ for each $\eps, R>0$,
where
$$D_{\eps,R}:=\left\{\Vphi\in\C_+:\norm{\Vphi}\leq R, x_i\geq \eps\;\forall i;\vec x:=\Vphi(0)\right\}.$$
}\end{asp}

\begin{thm}\label{thm4.1}
Assume  Assumptions  {\rm\ref{asp1}}, {\rm\ref{asp-2}},
{\rm\ref{a.extn}}, and {\rm\ref{asp-ginverse}} hold.
For any $p<p_0$ with $p_0$ being a
sufficiently small constant   $($as given in Lemma {\rm\ref{LV}}$)$,
and any initial value $\Bphi\in\C^{\circ}_+$, we have
\begin{equation}
\lim_{T\to\infty}\frac 1T\int_0^T\Ephi \bigwedge_{i=1}^n \norm{X_{i,t}}^{p}dt=0,
\end{equation}
where $\bigwedge_{i=1}^n x_i:=\min_{i=1,\dots,n}\{x_i\}$ and $\vec X_t=:(X_{1,t},\dots,X_{n,t})$.
\end{thm}

With additional technical conditions, we can determine which species goes extinct, and which persists.
First,
we define
  random normalized occupation measures
\begin{equation}\label{rcm}
\wdt\Pi_t(\cdot):=\dfrac1t\int_0^t\1_{\{\vec X_s\in\cdot\}}ds,\,t>0.
\end{equation}
For any initial condition $\Bphi\in\C_+$,
denote the weak$^*$-limit set of the family $\left\{\wdt \Pi_t(\cdot), t\geq 1\right\}$ by $\U=\U(\omega)$.

\begin{asp}\label{a.extn2}
{\rm
	Assume one
of the following conditions hold.
	\begin{itemize}
		\item Assumption \ref{asp-2}(a) holds and	
there exist  constants $p_2>0$ and $B_1>B_2>0, B_0>0$, $B_3>0$ such that for any $\Vphi\in \C_+,\vec x:=\Vphi(0)$
\begin{equation}\label{asp4-3a}
\begin{aligned}
(1+\vec c^\top\vec x)^{p_2}&\left(\dfrac{\sum_{i=1}^n c_ix_if_i(\Vphi)}{1+\vec c^\top \vec x}-\dfrac12\dfrac{\sum_{i,j=1}^n\sigma_{ij}c_ic_jx_ix_jg_i(\Vphi)g_j(\Vphi)}{(1+\vec c^\top \vec x)^2}
\right)
\\&\leq B_0-B_1(1+\vec c^\top\vec x)^{p_2}h(\vec x)+B_2\int_{-r}^0(1+\vec c^\top\Vphi(s))^{p_2}h\big(\Vphi(s)\big)\mu(ds),
\end{aligned}
\end{equation}
and
\begin{equation}\label{asp4-3b}
(1+\vec c^\top\vec x)^{2p_2}\sum_{i=1}^n g_i^2(\Vphi)\leq B_3(1+\vec c^\top\vec x)^{p_2}h(\vec x)+B_3\int_{-r}^0(1+\vec c^\top\Vphi(s))^{p_2}h\big(\Vphi(s)\big)\mu(ds).
\end{equation}
\item Assumption {\rm\ref{asp-2}(b)} is satisfied, and \eqref{asp4-3a} and \eqref{asp4-3b} hold with $h,\mu$ replaced by $h_1,\mu_1$.
\end{itemize}
}
\end{asp}

\begin{asp}\label{a.extn3}
{\rm
Let $S$ be a family of subsets $I$ satisfying Assumption \ref{a.extn}. We assume either that $S^c:=2^{\{1,\dots,n\}}\setminus S$ is empty, where $2^{\{1,\dots,n\}}$ denotes the family of all subsets of $\{1,\dots,n\}$,
or
$$\max_{i=1,\dots,n}\left\{\lambda_i(\nu)\right\}>0\text{ for any }\nu\in\conv(\displaystyle\cup_{J\notin S}\M^{J,\circ}).$$
}
\end{asp}

\begin{thm}\label{thm4.2}
Suppose that Assumptions {\rm\ref{asp1}}, {\rm\ref{a.extn}}, {\rm\ref{asp-ginverse}}, {\rm\ref{a.extn2}}, and {\rm\ref{a.extn3}} are satisfied.
Then for any $\Bphi\in\C^{\circ}_+$,
\begin{equation}\label{e0-thm4.2}
\sum_{I\in S} P_{\Bphi}^I=1,\quad P_{\Bphi}^I>0,
\end{equation}
where for $\Bphi\in\C^{\circ}_+$,
$$
\begin{aligned} 
P_{\Bphi}^I:=\PPphi\bigg\{\U(\omega)&\subset\conv(\M^{I,\circ})\,\text{ and }
\\&\lim_{t\to\infty}\dfrac{\ln X_i(t)}t\subset\left\{\lambda_i(\pi):\pi\in\conv(\M^{I,\circ})\right\}, i\in I^c\bigg\}.
\end{aligned}
$$
In the above, $\lim_{t\to\infty} x(t)$ can be understood as the set of limit points of $x(\cdot)$ as $t\to\infty$.
\end{thm}

\begin{rem}
From a dynamic system point of view, we have the following observations; see also \cite{HN18}.
\begin{itemize}
\item Assumption \ref{a.extn} states the existence of an attracting subspace on the boundary which normally results in extinction.
\item Assumption \ref{a.extn3} is a technical condition ensuring that the interior of the attracting subspace in Assumption \ref{a.extn} is an attractor in that subspace.
\item Assumption \ref{a.extn2} is a condition to control the volatility of the diffusion part while the nondegenaracy of the diffusion part due to Assumption \ref{asp-ginverse} leads to the accessibility to the boundary from any interior point.
\end{itemize}

\end{rem}

\section{Technical Results}\label{sec:key}

\subsection{Well-posedness of the problem}
The well-posedness of the problem \eqref{Kol-Eq} and some basic properties have been studied in the first part \cite{NNY21}.
We restate some results, while the proofs are referred to \cite{NNY21}.
The following series of results provide the estimating of the infinitesimal operator $\Lom V_{\Vrho}$, the well-posedness of the problem, the ``local" compactness of the solution, the regularity of the solution and the continuity on the initial data, respectively.
The formula of the infinitesimal operator and the functional It\^o formula are referred to the first part \cite{NNY21}.

\begin{lm}\label{LV}
	For any $\gamma<\gamma_b$ and $p_0>0$, $\Vrho=(\rho_1,\dots,\rho_n)\in\R^n$ satisfying
	\begin{equation}\label{cd-p0}
	\abs{\Vrho}<\min\left\{\dfrac{\gamma_b}2, \dfrac1n, \frac{\gamma_b}{4\sigma^*}\right\}\;
	\text{and}\;p_0<\min\left\{1,\dfrac{\gamma_b}{8n\sigma^*}\right\},
	\end{equation}
	where $\sigma^*:=\max\{\sigma_{ij}: 1\leq i,j\leq n\}$,
	let
	$$\displaystyle V_{\Vrho}(\Vphi):=\Big(1+\vec c^\top\vec x\Big)\prod_{i=1}^nx_i^{\rho_i}\exp\left\{A_2\int_{-r}^0\mu(ds)\int_s^0 e^{\gamma(u-s)}h\big(\Vphi(u)\big)du\right\}.$$
	Then,
	we have
	\begin{equation}\label{est-LV}
	\begin{aligned}
	\Lom V_{\Vrho}^{p_0}(\Vphi)\leq &p_0V_{\Vrho}^{p_0}(\Vphi)\Bigg[A_0 \1_{\{|\vec x|<M\}}-\gamma_0-Ah(\vec x)
	\\&-A_2\gamma \int_{-r}^0\mu(ds)\int_s^0e^{\gamma(u-s)}h\big(\Vphi(u)\big)du-\dfrac{\gamma_b}2\sum_{i=1}^n\Big(\abs{f_i(\Vphi)}+g_i^2(\Vphi)\Big)
	\Bigg],
	\end{aligned}
	\end{equation}
	where $\vec x:=\Vphi(0)$ and $A$ is a positive number satisfying $A<A_1-A_2\int_{-r}^0 e^{-\gamma s}\mu(ds)$.
	Recall that $\vec c$, $M$, $A_0$, $A_1$, $A_2$, $\gamma_0$, $\gamma_b$, $h(\cdot)$, $\mu(\cdot)$ are defined in Assumption {\rm\ref{asp1}(3)}.
\end{lm}

\begin{thm}\label{existence}
	For any initial condition $\Bphi\in\C_+$, there exists a unique global solution of \eqref{Kol-Eq}. It remains in $\C_+$ $($resp. $\C_+^{\circ})$, provided $\Bphi\in\C_+$ $($resp. $\Bphi\in\C_+^{\circ})$.
	Moreover, for any $p_0,\Vrho$ satisfying
	condition \eqref{cd-p0}, we have
	\begin{equation}\label{EV-bounded-ne}
	\Ephi V_{\Vrho}^{p_0}(\vec X_t)\leq V_{\Vrho}^{p_0}(\Bphi)e^{A_0p_0t}.
	\end{equation}
	In addition, if $\rho_i\geq 0,\forall i$, then
	\begin{equation}\label{EV-bounded}
	\Ephi V_{\Vrho}^{p_0}(\vec X_t)\leq V_{\Vrho}^{p_0}(\Bphi)e^{-\gamma_0p_0t}+\bar M_{p_0,\Vrho},
	\end{equation}
	where $$\bar M_{p_0,\Vrho}:=\dfrac {A_0}{\gamma_0}\sup_{\Vphi\in \C_{V,M}}V_{\Vrho}^{p_0}(\Vphi)<\infty\text{ provided }\rho_i\geq 0\;\forall i,$$
	and
	$\C_{V, M}=\{\Vphi\in\C_+: A_2\gamma \int_{-r}^0\mu(ds)\int_s^0e^{\gamma(u-s)}h\big(\Vphi(u)\big)du\leq A_0 \text{ and } |\vec x|\leq M\}$.
\end{thm}

\begin{lm}\label{X-bounded}
	For any $R_1>0$, $T>r$, $\eps>0$, there exists an $R_2>0$ such that
	\begin{equation*}
	\PPphi\Big\{\norm{\vec X_t}\leq R_2,\;\forall t\in [r,T]\Big\}>1-\eps,
	\end{equation*}
	for any initial point $\Bphi$ satisfying $V_{\vec 0}(\Bphi)<R_1$, where $V_{\vec 0}$ is defined as in Lemma {\rm\ref{LV}} corresponding to $\Vrho=\vec 0=(0,\dots,0).$
\end{lm}

\begin{lm}\label{Holder}
	There is a
	sufficiently small $\alpha>0$
	such that
	for any $R>0$ and $\eps>0$, there exists $R_3=R_3(R,\eps)>0$ satisfying
	\begin{equation}
	\text {if }\norm{\Bphi}\leq R\text{ then }\PPphi\left\{\|\vec X_{t}\|_{2\alpha}\leq R_3\;\forall t\in[r,3r]\right\}\geq1-\frac{\eps}2.
	\end{equation}
	As a consequence,
	for any $R>0$ and $\eps>0$, there exists an $R_4=R_4(\eps,R)>0$  satisfying that
	\begin{equation}
	\text {if }V_{\vec 0}(\Bphi)\leq R\text{ then }\PPphi\left\{\|\vec X_{t}\|_{2\alpha}\leq R_4\;\forall t\in[2r,3r]\right\}\geq1-\eps.
	\end{equation}
\end{lm}

\begin{prop}\label{theorem-2.4}
	The following results hold.
	\begin{itemize}
		\item [\rm(i)]
		Let $\rho_1^{(3)}$ be a fixed constant satisfying $0<\rho_1^{(3)}<\min\left\{\frac{\gamma_b}2, \frac1n, \frac{\gamma_b}{4\sigma^*}\right\}$.
		For any $T>r$ and $m>0$ there exists
		a finite constant $K_{m,T}$ such that
		$$\Ephi\norm{ X_{i,t}}^{p_0\rho_1^{(3)}}\leq K_{m,T}\phi_i^{p_0\rho_1^{(3)}}(0),\;\forall t\in [r,T], i=1,\dots,n,$$
		given that $$\abs{\Bphi(0)}+\int_{-r}^0\mu(ds)\int_s^0e^{\gamma(u-s)}h\big(\Bphi(u)\big)du<m,$$
		where $\vec X_t=:(X_{1,t},\dots,X_{n,t})$ and $\Bphi=:(\phi_1,\dots,\phi_n)$ is the initial value.
		\item[\rm(ii)] For any $T>r$, $\eps>0$, $R>0$, there exists an $\eps_1>0$ such that
		\begin{equation}\label{cont-ini}
		\PP\left\{\big\|{\vec X^{\Bphi_1}_T-\vec X^{\Bphi_2}_T}\big\|\leq \eps\right\}\geq 1-\eps\text{ whenever }V_{\vec 0}(\Bphi_i)<R, \norm{\Bphi_1-\Bphi_2}\leq \eps_1.
		\end{equation}
		Moreover, the solution $(\vec X_t)$ has the Feller property in $\C_+$.
	\end{itemize}
\end{prop}

\subsection{Random occupation measures: tightness and convergence properties}
Next, we deal with the tightness and uniform integrability of the random normalized occupation measures, which are defined by
\begin{equation*}
\wdt\Pi_t(\cdot):=\dfrac1t\int_0^t\1_{\{\vec X_s\in\cdot\}}ds,\,t>0.
\end{equation*}

\begin{lm}\label{lm4.5}
	Suppose that Assumptions  {\rm\ref{asp1}} and {\rm\ref{a.extn2}} are satisfied. Then the following results hold
	\begin{itemize}
		\item There is  a $\wdt G>0$ such that $\text{ for all }\Bphi\in\C_+$
		\begin{align*}
		\PPphi\Bigg\{\limsup_{T\to\infty}&\dfrac1T\int_0^T \left[\left(1+\sum_i c_iX_i(t)\right)^{p_2}h(\vec X(t))\right.\\ &\left.+\int_{-r}^0\left(1+\sum_{i}c_iX_i(t+s)\right)^{p_2}h(\vec X(t+s))\mu(ds)\right]dt\leq \wdt G\Bigg\}=1,
		\end{align*}
		where $p_2$ is as in Assumption {\rm\ref{a.extn2}}.
		\item
		Suppose that we have a sample path  of $\vec X$ satisfying
		$$
		\begin{aligned}
		\limsup_{T\to\infty}\dfrac1T\int_0^T &\left[\left(1+\sum_i c_iX_i(t)\right)^{p_2}h(\vec X(t))\right.\\ &\qquad\left.+\int_{-r}^0\left(1+\sum_{i}c_iX_i(t+s)\right)^{p_2}h(\vec X(t+s))\mu(ds)\right]dt\leq \wdt G,
		\end{aligned}
		$$
		and that there is a sequence $(T_k)_{k\in\N}\subset\R^n_+$ such that
		$\lim_{k\to\infty}T_k=\infty$ and
		$\left\{\wdt\Pi_{T_k}(\cdot)\right\}_{k\in\N}$ converges weakly to an invariant probability measure $\pi$ of $\vec X$
		when $k\to\infty$ .
		Then for this sample path, we have
		$$\int_{\C_+}K(\Vphi)\wdt\Pi_{T_k}(d\Vphi)\to \int_{\C_+}K(\Vphi)\pi(d\Vphi),$$
		for any continuous function $K:\C_+\to\R$ satisfying $\forall\Vphi\in \C_+, 0<p<p_2,$
		$$|K(\Vphi)|<C_K\left[\left(1+\vec c^\top\vec x\right)^{p}h(\vec x)+\int_{-r}^0\left(1+\vec c^\top\Vphi(s)\right)^{p}h(\Vphi(s))\mu(ds)\right],$$
		with $C_K$ being a positive constant.
		\item There is a constant $\hat K_1>1$ such that
		\begin{equation}\label{e1-lm4.7}
		\PPphi\left\{\liminf_{t\to\infty} \dfrac{1}t\int_0^t\1_{\{\|\vec X_s\|\leq \hat K_1\}}ds\geq\dfrac12\right\}=1,\,\Bphi\in\C_+.
		\end{equation}
		Moreover, for any $\eps_1$ and $\eps_2>0$, there is a $\beta>0$ such that
		for each $i=1,\dots,n$,
		\begin{equation}\label{e2-lm4.7}
		\PPphi\{X_i(t)>\beta\,,\forall\, t\in[0,n^*T_e]\}>1-\eps_1\,\text{ if }\, \Bphi\in\C_+,V_{\vec 0}(\Bphi)\leq \hat K_1, \phi_i(0)>\eps_2,
		\end{equation}
		where $n^*$ and $T_e$ are as in Lemma \ref{lm4.2}.
	\end{itemize}
\end{lm}

\begin{proof}
	Consider the first assertion.
	We obtain from the functional It\^o formula and \eqref{asp4-3a} that
	\begin{equation}\label{lm4.5-e2}
	\begin{aligned}
	&\dfrac{\left(1+\vec c^\top \vec X(t)\right)^{p_2}-\left(1+\vec c^\top \vec X(0)\right)^{p_2}}t\\
	&\qquad\leq
	B_0- \dfrac{B_1}t\int_0^t (1+\vec c^\top \vec X(s))^{p_2}h(\vec X(s))ds\\
	&\qquad\quad+\dfrac{B_2}t\int_0^t ds\int_{-r}^0(1+\vec c^\top \vec X(u+s))^{p_2}h(\vec X(u+s))\mu(du)
	+\dfrac{\bar L(t)}t\\
	&\qquad\leq B_0+\dfrac{B_2}t\int_0^r ds\int_{-r}^0(1+\vec c^\top \vec X(u+s))^{p_2}h(\vec X(u+s))\mu(du)\\
	&
	\qquad\quad- \dfrac{B_1-B_2}t\int_0^t (1+\vec c^\top \vec X(s))^{p_2}h(\vec X(s))ds
	+\dfrac{\bar L(t)}t,
	\end{aligned}
	\end{equation}
	where $\bar L(t)$ is the diffusion part of $\left(1+\vec c^\top \vec X(t)\right)^{p_2}$.
	Due to \eqref{asp4-3b}, we have the following estimate for the quadratic variation $\langle \bar L., \bar L.\rangle_t$ of $\bar L(t)$
	$$
	\begin{aligned}\langle \bar L., \bar L.\rangle_t\leq& B_3\int_0^t(1+\vec c^\top \vec X(s))^{p_2}h(\vec X(s))ds\\
	&+B_3\int_0^t ds\int_{-r}^0(1+\vec c^\top \vec X(u+s))^{p_2}h(\vec X(u+s))\mu(du)\\
	\leq&
	B_3\int_0^r ds\int_{-r}^0(1+\vec c^\top \vec X(u+s))^{p_2}h(\vec X(u+s))\mu(du)
	\\
	&+
	2B_3\int_0^t(1+\vec c^\top \vec X(s))^{p_2}h(\vec X(s))ds.
	\end{aligned}
	$$
	It follows from the strong law of large numbers for local martingales that
	\begin{equation}\label{lm4.5-e3}
	\limsup_{t\to\infty}\left( -\frac{B_1-B_2}{2t}\int_0^t (1+\vec c^\top \vec X(s))^{p_2}h(\vec X(s))ds
	+\dfrac{\bar L(t)}t\right)
	\leq 0 \text{ a.s.}
	\end{equation}
	Applying \eqref{lm4.5-e3} and noting
$\liminf_{t\to\infty}\dfrac{(1+\vec c^\top \vec X(t))^{p_2}-(1+\vec c^\top \vec X(0))^{p_2}}t\geq 0$ to \eqref{lm4.5-e2}, we have
	$$
	\limsup_{t\to\infty} \dfrac1t\int_0^t (1+\vec c^\top \vec X(s))^{p_2}h(\vec X(s))ds\leq \dfrac{2B_0}{B_1-B_2}\a.s
	$$
	Therefore, the first part of Lemma \ref{lm4.5} is proved.
	
	The proof of the second part is the same as
	that of
	\cite[Lemma 3.5]{NNY21}
	and is omitted.
	The
		proof of the third part can be found in \cite[Proof of Lemma 5.5]{HN18}.
\end{proof}

Compared to \cite[Lemma 5.7]{HN18}, it is much more difficult to prove the tightness and characterize the limit of the  normalized occupation measures in this setting because  $\C$ is an infinite dimensional space.

\begin{lm}\label{lm4.6}
Let Assumptions {\rm\ref{asp1}} and {\rm\ref{a.extn2}} be satisfied.
	For any initial condition $\Bphi\in\C_+$,
	the family $\left\{\wdt \Pi_t(\cdot), t\geq 1\right\}$ is tight in $\C_+$,
	and its weak$^*$-limit set, denoted by $\U=\U(\omega)$,
	is a family of invariant probability measures of $\vec X_t$ with probability 1.
\end{lm}

\begin{proof}
	For simplicity of notation, denote $\wdt h(\vec x)=(1+\vec c^\top \vec x)^{p_2}h(\vec x)$. We have
	\begin{equation}\label{e-l6-eq1}
	\begin{aligned}
	\int_0^T &\left[\int_{-r}^0\wdt h(\vec X_t(u))du\right]dt=\int_0^T \left[\int_{-r}^0\wdt h(\vec X(t+u))du\right]dt\\
	&\qquad\leq  \int_{-r}^0\left[\int_{-r}^T \wdt h(\vec X(t))dt\right]du
	\leq r\int_0^T \wdt h(\vec X(t))dt+r\int_{-r}^0\wdt h(\vec X(u))du.
	\end{aligned}
	\end{equation}
	Hence, by Lemma \ref{lm4.5} and \eqref{e-l6-eq1}, we deduce that
	\begin{equation}\label{e-l6-eq2}
	\lim_{T\to\infty}\int_0^T \left[\wdt h(\vec X(t))+\int_{-r}^0 \wdt h(\vec X_t(u))du\right]dt\leq\wdt G (1+r).
	\end{equation}
	Define
	$$\wdt C_R:=\left\{\Vphi\in\C_+^\circ: \wdt h(\vec x)+\int_{-r}^0\wdt h(\Vphi(u))du<R,\vec x=\Vphi(0)\right\}.$$
	A consequence of \eqref{e-l6-eq2} is that for any $\eps>0$, there is an $R>0$ such that
	\begin{equation}\label{e-l6-eq3}
	\lim_{T\to\infty}\frac 1T\int_0^T \1_{\{\vec X_t\in\wdt C_R\}}dt>1-\eps.
	\end{equation}
	It is easily seen
	that
	$$
	\sup_{\Vphi\in\wdt C_R}V_{\vec 0}^{p_0}(\Vphi)<\infty\;\forall R>0.
	$$
	By Lemma \ref{Holder}, there is a compact set $\K:=\{\Vphi:\|\Vphi\|_{2\alpha}<R_4\}$ of $\C^\alpha$, for some $\alpha>0$ and $R_4=R_4(\eps,R)$ such that
	\begin{equation}\label{e-l6-eq4}
	\PPphi\{\vec X_t\in\K\;\forall t\in[2r,3r]\}\geq 1-\eps, \quad\Bphi\in\wdt C_R.
	\end{equation}
	Let
	$Y_k:=\1_{\{\vec X_{kr}\in\K\}}$ for $k\in \N$,
	then $\sum_{l=1}^k Y_l=A_k+M_k$
	with $$A_k:=\sum_{l=1}^k\E\left[Y_l|\F_{(l-1)r}\right]\;;\;M_k:=Y_0+\sum_{l=1}^k\left(Y_l-\E\left[Y_l|\F_{(l-1)r}\right]\right).$$
	By strong law of large number for martingales, it is easily seen that
	\begin{equation}\label{e-l6-eq5}
	\lim_{k\to\infty}\frac{M_k}{k}=0\a.s
	\end{equation}
	To proceed, we estimate $\E \left[Y_l|\F_{(l-1)r}\right]$ for $l\geq 2$. In the event
	$\left\{\frac 1r\int_{(l-2)r}^{(l-1)r}\1_{\{\vec X_t\in \wdt C_R\}}dt>0\right\},$
	$\vec X_t\in \wdt C_R$ for some $t\in[(l-2)r,(l-1)r]$ and then, by the strong Markov property of $\vec X_t$ and \eqref{e-l6-eq4}, we have
	\begin{equation*}
	\E\left[Y_l\Big|\frac 1r\int_{(l-2)r}^{(l-1)r}\1_{\{\vec X_t\in \wdt C_R\}}dt>0\right]\geq 1-\eps.
	\end{equation*}
	Thus, owing to $\frac 1r\int_{(l-2)r}^{(l-1)r}\1_{\{\vec X_t\in\wdt C_R\}}dt\leq 1$, we can write
	\begin{equation*}
	\E\left[Y_l|\F_{(l-1)r}\right]\geq \frac {1-\eps}r\int_{(l-2)r}^{(l-1)r}\1_{\{\vec X_t\in \wdt C_R\}}dt\;\a.s
	\end{equation*}
	As a result,
	\begin{equation*}
	\frac{A_k}k\geq \frac {1-\eps}{kr}\int_0^{(k-1)r}\1_{\{\vec X_t\in\wdt C_R\}}dt\a.s,
	\end{equation*}
	which together with \eqref{e-l6-eq3} implies that
	\begin{equation}\label{e-l6-eq6}
	\liminf_{k\to\infty}\frac{A_k}k\geq (1-\eps)^2\a.s
	\end{equation}
	We deduce from \eqref{e-l6-eq5} and \eqref{e-l6-eq6} that
	\begin{equation}\label{e-l6-eq7}
	\liminf_{k\to\infty}\frac{\sum_{l=1}^kY_l}k\geq 1-2\eps\a.s
	\end{equation}
	We have the following
	estimate
	\begin{equation}\label{e-l6-eq8}
	1-\1_{\{\vec X_{lr}\in\K\text{ and }\vec X_{(l+1)r}\in\K\}}\leq (1-Y_l)+(1-Y_{l+1}).
	\end{equation}
	Due to \eqref{e-l6-eq7} and \eqref{e-l6-eq8}, we get
	\begin{equation}\label{e-l6-eq9}
	\lim_{k\to\infty}\frac{\sum_{l=1}^k \1_{\{\vec X_{lr}\in\K\text{ and }\vec X_{(l+1)r}\in\K\}}}{k}\geq 1-4\eps\a.s
	\end{equation}
	By Lemma \ref{Holder}, it is easy to show that there is a compact set $\wdt\K$ of $\C$ such that
	\begin{equation}\label{e-l6-eq10}
	\vec X_t\in\wdt \K,\forall t\in[lr,(l+1)r]\text{ if }\vec X_{lr}\in\K\text{ and }\vec X_{(l+1)r}\in\K.
	\end{equation}
	A consequence of \eqref{e-l6-eq9} and \eqref{e-l6-eq10} is that
	\begin{equation}\label{e-16-eq11}
	\lim_{T\to\infty}\frac 1T\int_0^T \1_{\{\vec X_t\in\wdt\K\}}\geq 1-4\eps\a.s
	\end{equation}
	As a result, the tightness of $\{\wdt\Pi_t(\cdot)\}$ is obtained.

	Moreover, from \eqref{e-16-eq11}, we obtain a family of compact sets $\{\K_\eps\subset\C_+,\eps\in(0,1)\}$ such that
	for any $\eps>0$, $\PPphi\left\{\U(\omega)\in \U_\K\right\}=1$,
	where $\U_\K$ is the set of probability measures on $\C$ satisfying  $ \pi(\K_\eps)>1-\eps \text{ for all } \eps\in(0,1).$
	From the definition of $\U_\K$, there exists a countable family $\{v_k\}$
of
bounded and continuous functions from $\C$ to $\R$ such that
	for any bounded and continuous function $v$ and  measure $\pi\in \U_\K$, we have
	\begin{equation}\label{e-16-eq12}
	\int v(\Vphi)\pi(d\Vphi)=\lim_{k_n\to\infty}\int v_{k_n}(\Vphi)\pi(d\Vphi).
	\end{equation}
Using the standard arguments in \cite[Proof of Theorem 4.2]{EHS15},
we can show that outside a null set,  any weak limit of $\{\wdt\Pi_t(\cdot)\}$, denoted by $\wdt\pi$, satisfies
	\begin{equation}\label{e-16-eq13}
	\int_{\C_+}\wdt\pi(d\Vphi)\int P(t,\Vphi,\boldsymbol{\psi})v_k(\boldsymbol{\psi})=\int v_k(\Vphi)\wdt\pi(d\Vphi).
	\end{equation}
	 From \eqref{e-16-eq12} and \eqref{e-16-eq13}, we have that outside a null set, for any  bounded and continuous function $v$,
	\begin{equation}
	\int_{\C_+}\wdt\pi(d\Vphi)\int P(t,\Vphi,\boldsymbol{\psi})v(\boldsymbol{\psi})=\int v(\Vphi)\wdt\pi(d\Vphi).
	\end{equation}
	The lemma is thus proved.
\end{proof}

\section{Extinction}\label{sec:ext}
Following the development in the last section, this section focuses on obtaining the criteria of extinction.
To start, we have the following
Lemma, whose proof is easily obtained by modifying the proof of \cite[Lemma 5.1]{HN18}.

\begin{lm}\label{lm4.1}
For any $\pi\in\M$ and $i\in I_\pi$, we have
$\lambda_i(\pi)=0.$
\end{lm}

The intuition behind Lemma \ref{lm4.1} is clear. If the process evolves in the interior of the support of an ergodic invariant measure $\mu$, it will eventually approach the ``stationary" state with probability measure $\mu$ and it cannot
grow or decay exponentially fast.

It is shown in \cite[Lemma 4]{SBA11}, by the min-max
principle, that condition \eqref{ae3.2} is equivalent to the existence of
$0<\hat \alpha_i<p_0, i\in I$
such that
$$\inf_{\nu\in\partial \M^I}\sum_{i\in I}\hat \alpha_i\lambda_i(\nu)>0.$$
Thus, there is an $\alpha_*>0$ sufficiently small such that
\begin{equation}\label{e3.2}
\begin{aligned}
\inf_{\nu\in\partial \M^I}\sum_{i\in I}\hat \alpha_i\lambda_i(\nu)-\alpha_*\max_{i\in I^c}\{\lambda_i(\nu)\}>0.
\end{aligned}
\end{equation}
In view of \eqref{e3.2}, \eqref{ae3.1}, and Lemma \ref{lm4.1}, there is a $\kappa_e>0$ such that for any $\nu\in\M^I$,
\begin{equation}\label{e3.3}
\sum_{i\in I}\hat \alpha_i\lambda_i(\nu)-\alpha_*\max_{i\in I^c}\left\{\lambda_i(\nu)\right\}>3\kappa_e.
\end{equation}
Now, denote
\begin{equation*}
\begin{aligned}
Q_{\vec 0}(\Vphi)=&A_2 h(\vec x)\int_{-r}^0 e^{-\gamma s}\mu(ds)-A_2\int_{-r}^0h\big(\Vphi(s)\big)\mu(ds)
\\&-A_2\gamma \int_{-r}^0\mu(ds)\int_s^0e^{\gamma(u-s)}h\big(\Vphi(u)\big)du
\\&+\dfrac{\sum_{i=1}^n c_ix_if_i(\Vphi)}{1+\vec c^\top\vec x}-\dfrac 12\sum_{i,j=1}^n  \dfrac{c_ic_j\sigma_{ij}x_ix_jg_i(\Vphi)g_j(\Vphi)}{\Big(1+\vec c^\top\vec x\Big)^2},
\end{aligned}
\end{equation*}
and let $n^*$ be a sufficient large integer such that
\begin{equation}\label{e:n*}
\gamma_0 (n^*-1)- A_0>0.
\end{equation}

\begin{lm}\label{lm4.2}
Suppose that Assumptions {\rm\ref{asp1}}, {\rm\ref{asp-2}}, and {\rm\ref{a.extn}} hold. Let $I\subset\{1,\dots, n\}$ satisfy Assumption {\rm\ref{a.extn}}.
Then there are $T_e\geq 0$ and $\delta_e>0$ such that for any $T\in[T_e,n^*T_e]$, $\Bphi\in\C^{\circ}_+\cap \C_{V,M}, \phi_i(0)<\delta_e, \forall i\in I^c$, we have
\begin{equation}\label{e3.4}
\begin{aligned}
\dfrac1T\int_0^T\Ephi Q_{\vec 0}(\vec X_t)dt&-\sum_{i\in I}\hat \alpha_i\dfrac1T\int_0^T\Ephi\left(f_i(\vec X_t)-\dfrac{\sigma_{ii}g^2_i(\vec X(t))}{2} \right)dt\\
&+\alpha_*\max_{i\in I^c}\left\{\dfrac1T\int_0^T\Ephi\left(f_i(\vec X_t)-\dfrac{\sigma_{ii}g^2_i(\vec X_t)}{2} \right)dt\right\}
\leq -\kappa_e,
\end{aligned}
\end{equation}
\end{lm}

\begin{proof}
The proof is similar to \cite[Lemma 4.2]{NNY21}.
First, using \eqref{e3.3}, we can prove that for any compact set $\K$, there exists a $T_\K>0$ such that for any $T>T_\K, \Bphi\in\C^{\circ}_+\cap \K$, we have
\begin{equation}
\begin{aligned}
\dfrac1T\int_0^T\Ephi Q_{\vec 0}(\vec X_t)dt&-\sum_{i\in I}\hat \alpha_i\dfrac1T\int_0^T\Ephi\left(f_i(\vec X_t)-\dfrac{\sigma_{ii}g^2_i(\vec X_t)}{2} \right)dt\\
&+\alpha_*\max_{i\in I^c}\left\{\dfrac1T\int_0^T\Ephi\left(f_i(\vec X_t)-\dfrac{\sigma_{ii}g^2_i(\vec X_t)}{2} \right)dt\right\}
\leq -2\kappa_e.
\end{aligned}
\end{equation}
Then, although the Feller property of $(\vec X_t)$ is not directly applied here because  a bounded set in an infinite dimensional space is not necessarily pre-compact,
we can overcome the difficulty by using Lemma \ref{Holder} and Proposition \ref{theorem-2.4}(ii). The detailed calculations are analogous to that of
\cite[Lemma 4.2]{NNY21}
and are omitted.
\end{proof}

\begin{lm}\label{lm4.3}
Suppose that Assumptions {\rm\ref{asp1}}, {\rm\ref{asp-2}}, and {\rm\ref{a.extn}} hold, and let $\hat \alpha_i,\alpha_*, \delta_e, n^*,T_e$
as in Lemma \ref{lm4.2}.
Then there is a $\theta\in(0,p_0)$ such that for any $T\in[T_e,n^*T_e]$ and $\Bphi\in\C^{\circ}_+\cap \C_{V,M}$ satisfying  $\phi_i(0)<\delta_e,$ $\forall i\in I^c$ one has
$$\Ephi \hat U_\theta(\vec X_T)\leq \exp\left(-\frac{1}{4}\theta \kappa_eT\right) \hat U_\theta(\Bphi),$$
where
$$
\begin{aligned}
\hat U_\theta(\Vphi):=&\sum_{i\in I^c}V_{\Vrho^{i,e}}^{\theta}(\Vphi)\\
=&\sum_{i\in I^c}\left[(1+\vec c^\top\vec x)\dfrac{x_i^{\alpha_*}}{\prod_{j\in I} x_j^{\hat \alpha_j}}\exp\Big\{A_2\int_{-r}^0\mu(ds)\int_s^0 e^{\gamma(u-s)}h\big(\Vphi(u)\big)du\Big\}\right]^\theta,
\end{aligned}
$$
and
$\Vrho^{i,e}=(\rho^{i,e}_1,\dots,\rho^{i,e}_n)$  and
$$\rho^{i,e}_j=\alpha_*\text{ if }j=i, \;\;\rho^{i,e}_j=-\hat \alpha_j\text{ if }j\neq i, j\in I\text{ and otherwise, }\rho^{i,e}_j=0.$$
\end{lm}

\begin{proof}
The argument to prove this Proposition is similar to
 that of
 \cite[Proposition 4.1]{NNY21}.
 For each $i\in I^c$, by making use of Lemma \ref{lm4.2}, there exists a $\theta>0$ such that for $T\in[T_e,n^*T_e]$, $\Bphi\in \C_+^\circ\cap\C_{V,M}$ with $\phi_i(0)<\delta_e$, we have
$$\Ephi V_{\Vrho^{i,e}}^\theta(\vec X_T)\leq \exp\left(-\frac{1}{4}\theta \kappa_eT\right) V_{\Vrho^{i,e}}^\theta(\Bphi).$$
Therefore, the Proposition follows from the definition of $\hat U_\theta$.
\end{proof}

\begin{prop}\label{prop4.1}
Under Assumptions {\rm\ref{asp1}}, {\rm\ref{asp-2}}, and {\rm\ref{a.extn}}.
For any $\eps>0$, there exists $\delta=\delta(\eps)>0$ such that
\begin{equation}
\PPphi\left\{\lim_{t\to\infty}\hat U_\theta(\vec X_t)=0\right\}\geq 1-\eps,\text{ for all }\Bphi \text{ satisfying }\hat U_\theta(\Bphi)<\delta.
\end{equation}
\end{prop}

\begin{proof}
Let
$$
\begin{aligned}
C_0&:=\sup_{ \Vphi\in\C^{\circ}_+}\left\{\dfrac{\prod_{i\in I} x_i^{\hat \alpha_i}}{(1+\vec c^\top\vec x)\exp\Big\{A_2\int_{-r}^0\mu(ds)\int_s^0 e^{\gamma(u-s)}h\big(\Vphi(u)\big)du\Big\}}: \vec x=\Vphi(0)\right\}\\
&<\infty,
\end{aligned}
$$
and
\begin{equation}\label{de}
d(\delta_e):=\dfrac{(\delta_e)^{\theta\alpha_*}}{C_0^\theta},
\end{equation}
where $\delta_e$ is as in Lemma \ref{lm4.2}.
By \eqref{est-LV}, we have
\begin{equation}\label{et3.1}
\Lom \hat U_\theta(\Vphi)\leq -\theta\gamma_0\hat U_\theta(\Vphi) \text{ if } \Bphi\in\C^\circ_+,\Vphi\notin\C_{V,M}.
\end{equation}
Because of \eqref{et3.1}, \eqref{EV-bounded-ne}, and Proposition \ref{prop4.1}, by the same  procedure as
\cite[Theorem 4.1]{NNY21}, we obtain that
\begin{equation}\label{e-p1-eq1}
\text{ if }\hat U_\theta(\Bphi) \leq d(\delta_\eps)\text{ then } \Ephi Z(1)\leq q_1 Z(0), \text { for some } q_1\in(0,1),
\end{equation}
where
$Z(k):=d(\delta_e)\wedge \hat U_\theta(\vec X_{kn^*T_e}), k\in\N$. The reader can also see \cite[Proof of Theorem 5.1]{HN18} for detailed calculations of this argument.

For each $m<d(\delta_e)$, define the stopping time
$$\beta_{m}:=\inf\{k\in\N:Z(k)\geq m\}.$$
By \eqref{e-p1-eq1},
\begin{equation}\label{e-p1-eq2}
\Ephi \1_{\{\beta_m> k\}}Z(k)\leq q_1^kZ(0).
\end{equation}
In view of \eqref{e-p1-eq2}, we have
\begin{equation*}
\PPphi\left\{\beta_m>k\right\}\geq 1- \eps,\forall k\in\N\text{ if }\hat U_\theta(\Bphi)<m\eps.
\end{equation*}
Hence,  letting $k\to\infty$
leads to
\begin{equation}\label{e-p1-eq7}
\PPphi\left\{\beta_m=\infty\right\}\geq 1-\eps\text{ if }\hat U_\theta(\Bphi)<m\eps.
\end{equation}
On the other hand, using \eqref{est-LV} again and
 by the definition of $\hat U_\theta$, we have
\begin{equation*}
\Lom \hat U_\theta(\Vphi)\leq A_0\theta\hat U_\theta(\Vphi).
\end{equation*}
Hence, by a standard argument as in \cite[Proof of Theorem 3.1]{NNY21},
we get
\begin{equation*}
\Ephi e^{-A_0\theta (t\wedge \zeta_u)}\hat U_\theta (X_{t\wedge \zeta_u})\leq \hat U_\theta (\Bphi),
\end{equation*}
where for each $u>0$
$$\zeta_u:=\inf\{t\geq 0: \hat U_\theta (\vec X_t)>u\hat U_\theta (\Bphi)\},$$
which implies that
\begin{equation*}
ue^{-A_0\theta t}\hat U_\theta(\Bphi)\PPphi\{\zeta_u>t\}\leq \hat U_\theta(\Bphi).
\end{equation*}
Thus, for any $u>0$,
\begin{equation}\label{e-p1-eq3}
\PPphi\{\zeta_u>n^*T_e\}\leq \frac{e^{A_0\theta n^*T_e}}{u}.
\end{equation}
Let $q_2,q_3\in (0,q_1)$, $q_2<q_3$, where $q_1$ is as in \eqref{e-p1-eq1}. We obtain from \eqref{e-p1-eq2} that
\begin{equation}\label{e-p1-eq4}
\PPphi\left\{\1_{\{\beta_m>k\}}Z(k)\leq Z(0)q_2^k\right\}\geq 1-\left(\frac {q_1}{q_2}\right)^k.
\end{equation}
We deduce from \eqref{e-p1-eq3}, the Markov property of $\vec X_t$, and \eqref{e-p1-eq4} that
\begin{equation}\label{e-p1-eq5}
\begin{aligned}
\PPphi&\left\{\1_{\{\beta_m>k\}}\sup_{s\in [kn^*T_e,(k+1)n^*T_e]}\hat U_{\theta}(\vec X_s)\leq Z(0)q_3^k\right\}\\
&\geq \left(1-\frac{e^{A_0\theta n^*T_e}}{\left(\frac{q_3}{q_2}\right)^k}\right)\PPphi\left\{\1_{\{\beta_m>k\}}Z(k)\leq Z(0)q_2^k\right\}\\
&\geq \left(1-\frac{e^{A_0\theta n^*T_e}}{\left(\frac{q_3}{q_2}\right)^k}\right)\cdot\left(1-\left(\frac{q_1}{q_2}\right)^k\right)\\
&\geq 1-\left(e^{A_0\theta n^*T_e}\left(\frac{q_2}{q_3}\right)^k+\left(\frac{q_1}{q_2}\right)^k\right),
\end{aligned}
\end{equation}
for any $k>k_0$, where $k_0$ satisfies
$$
e^{A_0\theta n^*T_e}<\left(\frac {q_3}{q_2}\right)^{k_0}.
$$
Since
$$\sum_{k=k_0}^\infty\left(e^{A_0\theta n^*T_e}\left(\frac{q_2}{q_3}\right)^k+\left(\frac{q_1}{q_2}
\right)^k\right)<\infty,$$
by the Borel-Cantelli Lemma, we deduce from \eqref{e-p1-eq5} that
\begin{equation}\label{e-p1-eq6}
\PPphi\left\{\lim_{k\to\infty}\1_{\{\beta_m>k\}}\sup_{s\in [kn^*T_e,(k+1)n^*T_e]}\hat U_\theta (\vec X_s)=0\right\}=1.
\end{equation}
Combining \eqref{e-p1-eq6} and \eqref{e-p1-eq7} implies that
\begin{equation*}
\PPphi\left\{\lim_{s\to\infty}\hat U_\theta (\vec X_s)=0\right\}\geq 1-\eps\text{ if }\hat U_\theta(\Bphi)<m\eps.
\end{equation*}
Then the proposition is proved.
\end{proof}

\begin{lm}\label{e-l4}
Assume Assumption {\rm\ref{asp1}} and {\rm\ref{asp-ginverse}} hold.
For any $\eps,R>0$, $\vec y^*=(y_1^*,\dots,y_n^*)\in \R_+^{n,\circ}:=\left\{\vec y\in \R^n: y_i>0\;\forall i\right\}$, $\delta>0$, $t^*\geq 2r$
\begin{equation}\label{e-l4-eq1}
\inf_{\Bphi\in D_{\eps,R}}\PPphi\left\{\vec X_{t^*}\in B_{\vec y^*,\delta}\right\}>0,
\end{equation}
where
$$B_{\vec y^*,\delta}:=\left\{\Vphi\in\C_+^\circ: \abs{\Vphi(s)-\vec y^*}<\delta;\forall s\in [-r,0]\right\}.$$
\end{lm}

\begin{proof}
To prove \eqref{e-l4-eq1}, we  modify slightly the proof of \cite[Lemma 3.8]{Ha09} as follows.
Let $\delta_1\in (0,\frac{\delta}2)$ be sufficiently small such that $\min\left\{\min_{i=1,\dots,n}y_i^*-\delta_1,\eps-\delta_1\right\}=:2\delta_2>2\delta_1$.
Define
$$D(t):=\abs{\vec X(t)-\vec k(t)}^2-(\delta_1/2)^2,$$
where $\vec k:[0,t^*]\to\R^n$ is continuously differentiable with Lipschitz constant at most $\frac{2(R+|\vec y^*|+\delta)}r$ satisfying
$$
k_i(t)\geq 2\delta_2\;\forall t\in[0,t^*],\;\forall i,
\text{ and }
\vec k(0)=\Bphi(0)-(\delta_1/2,0,\dots,0);\vec k(t)=\vec y^*,t\in[r,t^*].
$$
It is noted that
\begin{equation}\label{e-l4-eq4}
X_i(t)\geq \delta_2\;\forall i\text{ and }\abs{\vec X(t)}<2(R+|\vec y^*|+\delta),t\in[0,t^*]\text{ whenever }\abs{D(t)}\leq (\delta_1/4)^2.
\end{equation}
Hence, under Assumption \ref{asp-ginverse} for
 the diffusion coefficients and \eqref{e-l4-eq4}, we can mimic the remaining of proof in \cite[Lemma 3.8]{Ha09}(with $\vec k$ in place of $\vec h$) and obtain that
\begin{equation*}
\inf_{\Bphi\in D_{\eps,R}}\PPphi\left\{\vec X_{t^*}\in B_{\vec y^*,\delta}\right\}>0.
\end{equation*}
\end{proof}

\begin{thm}
Assume that Assumptions {\rm\ref{asp1}}, {\rm\ref{asp-2}}, {\rm\ref{a.extn}}, and {\rm\ref{asp-ginverse}} hold.
For any $p<p_0$
and any $\Bphi\in\C^{\circ}_+$, we have
\begin{equation}\label{e.extinction}
\lim_{T\to\infty}\frac 1T\int_0^T\Ephi \bigwedge_{i=1}^n \norm{X_{i,t}}^{p}dt=0,
\end{equation}
where $\bigwedge_{i=1}^n x_i=\min_{i=1,\dots,n}\{x_i\}$ and $\vec X_t=(X_{1,t},\dots,X_{n,t})$.
\end{thm}

\begin{proof}
It is clear that if $\lim_{t\to\infty}\hat U_\theta(\vec X_t)=0$ then $\vec X_t$ tends to the boundary of $\C_+$ as $t\to\infty$. Moreover,
we can choose suitable $y^*$ and $\delta_1$ such that $\forall \Vphi\in B_{y^*,\delta_1}$, $\hat U_\theta(\Vphi)$ is small enough.
Therefore,
in view of 
Proposition \ref{prop4.1}
and Lemma \ref{e-l4},
the probability that $\vec X_t$ tends to the boundary is positive for any initial data. As a consequence,
there is no invariant probability measure in $\C^\circ_+$.
Therefore, we can deduce that the weak$^*$-limit of $\Pi_t^{\Bphi}(\cdot)$ is a probability measure concentrated on $\partial \C_+$. By noting that the function
$\left(\bigwedge_{i=1}^n \varphi_i^{p}(0)\right)$, $p<p_0$ of variable $\Vphi$ satisfies the condition
\cite[(3.41)]{NNY21},
 the Theorem follows from \cite[Lemma 3.5]{NNY21}. \end{proof}

\begin{lm}\label{lm4.9}
Suppose that Assumption {\rm \ref{asp1}}, {\rm\ref{a.extn}}, and {\rm\ref{a.extn2}} are satisfied and let $I$ be the subset of $\{1,\dots,n\}$
in Assumption {\rm\ref{a.extn}}.
Then
 the following results hold:
\begin{itemize}
\item For any $\Bphi\in\C^{\circ}_+$,
$$\PPphi\Big\{\U(\omega)\subset\conv(\M^{I})\Big\}=\PPphi\Big\{\U(\omega)\subset \conv(\M^{I,\circ})\Big\}.$$
\item For any $m>0$, $\delta>0$, and $\eps>0$,
there is a $R>0$ such that
$$
\begin{aligned}
\PPphi&\bigg\{\U(\omega)\subset\conv(\M^{I,\circ})\,\text{ and }\\
&\qquad\lim_{t\to\infty}\dfrac{\ln X_i(t)}t\subset\left\{\lambda_i(\pi), \pi\in\conv(\M^{I,\circ})\right\}, i\in I^c\bigg\}\\
&>1-\eps,\quad\text{ for all }\Bphi\in\Delta^{m,\delta,R}_I,
\end{aligned}
$$
where
$$
\begin{aligned} 
\Delta^{m,\delta,R}_I:=&\bigg\{\Vphi\in\C^{\circ}_+: m\leq x_i\text{ for } i\in I,x_i<\delta\text{ for }i\in I^c\text{ and }\\
&\qquad V_{\vec 0}(\Vphi)<R,\vec x:=\Vphi(0)\bigg\}.
\end{aligned}
$$
\end{itemize}
\end{lm}

\begin{proof}
The proof of the first part is similar to \cite[Proof of Lemma 5.8]{HN18}.

We proceed to prove the second part.
By the third part of Lemma \ref{lm4.5}, there is a
$k_0>0$ such that
\begin{equation}\label{e2-lm4.8}
\PPphi\left\{\max_{i\in I^c}\{X_i(t)\}>k_0\,,\forall\, t\in[0,n^*T_e]\right\}>\frac{1}{2},\,\Bphi\in\A_I,
\end{equation}
where
$$\A_I=\{\Vphi\in\C_+: V_{\vec 0}(\Vphi)\leq \hat K_1, \,\max_{i\in I^c}\{x_i\}\geq1,\vec x:=\Vphi(0)\}.$$
It can be seen that
\begin{equation}\label{e:nu_supp}
\nu(\A_I)>0 \,\text{ for }\,\nu\in\M\setminus\M^{I}.
\end{equation}
As in Proposition \ref{prop4.1}, consider $U_e(\Vphi):=d(\delta_e)\wedge \hat U_\theta(\Vphi)$, where $\delta_e$ is defined in \eqref{de}.
By the definition of $ U_e(\cdot)$, there is a
$\delta>0$ sufficiently small such that
\begin{equation}\label{e3-lm4.8}
\sup_{\Vphi\in\Delta^{m,\delta,R}_I}\{U_e(\Vphi)\}\leq \dfrac{\eps}{2}\inf_{\Vphi\in\C^{\circ}_+,x_i\geq k_0, \text{ for some }i\in I^c}\{U_e(\Vphi)\}.
\end{equation}
In view of \eqref{e-p1-eq7}, we obtain if $ \Bphi\in\Delta^{m,\delta,R}_I$
\begin{equation*}
\PPphi\left\{U_e(\vec X(kn^*T_e))<\inf_{\Vphi\in\C^{\circ}_+,x_i\geq k_0, \text{ for some }i\in I^c}\{U_e(\Vphi)\},\,\text{ for all } k\in\N\right\}>1-\dfrac\eps2.
\end{equation*}
Thus
\begin{equation}\label{e4-lm4.8}
\PPphi\left\{\max_{i\in I^c}\{X_i(kn^*T_e)\}< k_0\,\text{ for all } k\in\N\right\}>1-\dfrac\eps2\, \text{ if } \Bphi\in\Delta^{m,\delta,R}_I.
\end{equation}

Now, we prove
\begin{equation}\label{e7-lm4.8}
\PPphi\left\{
\lim_{t\to\infty}\dfrac1t\int_0^t \1_{\{\vec X_s\in \A_I\}}ds=0
\right\}>1-\eps, \,\Bphi\in \Delta^{m,\delta,R}_I,
\end{equation}
by a contradiction argument. Assume that
there is a $\Bphi\in\Delta^{m,\delta,R}_I$ satisfying
\begin{equation}\label{e5-lm4.8}
\PPphi\left\{
\limsup_{t\to\infty}\dfrac1t\int_0^t \1_{\{\vec X_s\in \A_I\}}ds>0
\right\}>\eps.
\end{equation}
Then
\begin{equation}\label{e6-lm4.8}
\PPphi\{\tau_{\A_I}<\infty\}>\eps,
\end{equation}
where $\tau_{\A_I}=\inf\{t>0: \vec X_t\in\A_I\}$.
By the strong Markov property of $\{\vec X_t\}$,
it follows from \eqref{e2-lm4.8} and \eqref{e6-lm4.8}
that
$$
\PPphi\left(\{\tau_{\A_I}<\infty\}\bigcap\left\{\max_{i\in I^c}\{X_i(t)\}\geq k_0\,\text{ for } t\in[\tau_{\A_I},\tau_{\A_I}+n^*T_e]\right\}\right)>\frac{1}{2}\eps,
$$
which contradicts \eqref{e4-lm4.8} and hence, \eqref{e7-lm4.8} holds.

We observe that if for an $\omega\in\Omega$ and a sequence $\{t_j\}$ with $\lim_{j\to\infty}t_j=\infty$,
 $\wdt \Pi_{t_j}(\cdot)$ converges weakly to an invariant probability of the form
$\pi_0=(1-q)\pi_1+q\pi_2$
 with $q\in[0,1]$,
 $\pi_1\in\conv(\M^{I})$, and $\pi_2\in\conv(\M\setminus\M^{I})$,
 then by \eqref{e:nu_supp}
$$\limsup_{j\to\infty}\dfrac1{t_j}\int_0^{t_j} \1_{\{\vec X_s\in \A_I\}}ds\geq \pi_0(\A_I)\geq q\pi_2(\A_I).$$
This inequality combined with Lemma \ref{lm4.6}, \eqref{e:nu_supp}, and \eqref{e7-lm4.8} implies that $q=0$ and
$$\PPphi\left\{
\U(\omega)\subset\conv(\M^{I})
\right\}>1-\eps, \,\Bphi\in \Delta^{m,\delta,R}_I.
$$
The first claim of Lemma \ref{lm4.9} and the above estimates lead to
\begin{equation}\label{e8-lm4.8}
\PPphi\Big\{
\U(\omega)=\{\conv(\M^{I,\circ})\}
\Big\}>1-\eps, \Bphi\in\Delta^{m,\delta,R}_I.
\end{equation}
In view of Lemma \ref{lm4.5} and \eqref{e8-lm4.8}, we have for $\Bphi\in\Delta^{m,\delta,R}_I$ and for each $i=1,\dots,n$ that
\begin{equation}\label{e9-lm4.8}
\PPphi\left\{\lim_{t\to\infty}\dfrac1t\int_0^{t}\left(f_i(\vec X_s)-\dfrac{\sigma_{ii}g_i^2(\vec X_s)}2\right)ds\subset \left\{\lambda_i(\pi): \pi\in\conv(\M^{I,\circ})\right\}\right\}>1-\eps.
\end{equation}
On the other hand, it is easy to see
\begin{equation}\label{e6-lm4.9}
\PPphi\left\{\lim_{t\to\infty}\dfrac1t\int_0^tg_i(\vec X_s)dE_i(s)=0,\,i=1,\dots,n\right\}=1.
\end{equation}
The second claim of Lemma \ref{lm4.9} follows from \eqref{e9-lm4.8}, \eqref{e6-lm4.9}, and an application of the functional It\^o formula.
\end{proof}

With the above
Lemmas in hand, we can modify slightly the proof of \cite[Theorem 5.2]{HN18} to obtain the following Theorem.

\begin{thm}
Suppose that Assumptions {\rm\ref{asp1}}, {\rm\ref{a.extn}}, {\rm\ref{asp-ginverse}}, {\rm\ref{a.extn2}}, and {\rm\ref{a.extn3}} are satisfied.
Then for any $\Bphi\in\C^{\circ}_+$
\begin{equation}
\sum_{I\in S} P_{\Bphi}^I=1,\quad P_{\Bphi}^I>0,
\end{equation}
where for $\Bphi\in\C^{\circ}_+$,
$$
\begin{aligned} 
P_{\Bphi}^I:=\PPphi\bigg\{\U(\omega)&\subset\conv(\M^{I,\circ})\,\text{ and }\\
&\lim_{t\to\infty}\dfrac{\ln X_i(t)}t\in\left\{\lambda_i(\pi), \pi\in\conv(\M^{I,\circ})\right\}, i\in I^c\bigg\}.
\end{aligned}
$$
\end{thm}

\section{Applications}\label{sec:app}
This section
presents a number of applications of our main results
Theorems \ref{thm4.1} and \ref{thm4.2}.
We provide sufficient conditions for the extinction of several popular biological and ecological systems.
These results are complements of
the permanence characterization in the first part \cite{NNY21}
in
that excluding the critical cases, if the system is not permanent, the extinction will happen and vice versa.

\subsection{Stochastic delay Lotka-Volterra competitive model}\label{subsec:1}
Introduced in \cite{Lot25,Vol26} by Lotka and Volterra in 1926, the Lotka-Volterra model is one of the most important models in mathematical biology and
has been studied widely in
the literature.
The Lotka-Volterra competitive models are used to describe the dynamics of the species when
they live in proximity, share the same basic
resources,
and compete for food, habitat, territory, etc.
Because of the influences of many complex properties in real life, other terms (white noises, Markov switching, delayed time, etc.) are added to the original system to reflect better the phenomena.
Stochastic delay Lotka-Volterra competitive models
have also been
widely studied;
see, for example, \cite{Mao04,Liu17} and references therein.
However, there is no unified general framework to handle that
except the work \cite[Section 5.1]{NNY21}, which provided criteria for persistence.

For the case of two-dimensional competitive stochastic delay system, this kind model takes the form
\begin{equation}\label{5.1-eq0}
\begin{cases}
dX_1(t)=X_1(t)\left(a_1-b_{11}X_1(t)-b_{12}X_2(t)-\hat b_{11}X_1(t-r)-\hat b_{12}X_2(t-r)\right)dt\\
\hspace{2cm}+X_1(t)dE_1(t),\\[1ex]
dX_2(t)=X_2(t)\left(a_2-b_{21}X_1(t)-b_{22}X_2(t)-\hat b_{21}X_1(t-r)-\hat b_{22}X_2(t-r)\right)dt\\
\hspace{2cm}+X_2(t)dE_2(t),
\end{cases}
\end{equation}
where
$X_i(t)$ denotes the size of the species $i$ at time $t$;
$a_i>0$ represents the growth rate of the species $i$;
$b_{ii}>0$ stands for the intra-specific competition of the $i^{th}$ species;
$b_{ij}\geq 0$, ($i\neq j$) is the inter-specific competition;
$\hat b_{ij}> -b_{ii}$ ($i,j=1,2$) (i.e., $\hat b_{ij}$ can be negative);
$r$ is the delay time;
$(E_1(t),E_2(t))^\top=\Gamma^\top\vec B(t)$ with
$\vec B(t)=(B_1(t), B_2(t))^\top$ being a vector of independent standard Brownian motions and
$\Gamma$ being a $2\times 2$ matrix such that
$\Gamma^\top\Gamma=(\sigma_{ij})_{2\times 2}$ is a positive definite matrix.

As a complement of
\cite[Section 5.1]{NNY21} that provides
the conditions for the persistence,
 we continue to characterize the extinction to complete the long-time characterization in this paper.
Applying
our Theorems in Section \ref{sec:res},
we
have that $\lambda_i(\bdelta^*)=a_i-\dfrac{\sigma_{ii}}2, i=1,2$.
In view of Theorem \ref{thm4.1}, if $\lambda_1(\bdelta^*)<0$, (resp. $\lambda_2(\bdelta^*)<0$) there is no invariant probability measure on $\C^{\circ}_{1+}:=\{(\varphi_1,0)\in\C_+: \varphi_1(s)>0\;\forall s\in[-r,0]\}$ (resp. $\C^{\circ}_{2+}:=\{(0,\varphi_2)\in\C_+: \varphi_2(s)>0\;\forall s\in[-r,0]\}$).
By Lemma \ref{lm4.1},
we have
$$\lambda_i(\pi_i)=a_i-\dfrac{\sigma_{ii}}2-\int_{\C^{\circ}_{i+}}\left(b_{ii}\varphi_i(0)+\hat b_{ii}\varphi_i(-r)\right)\pi_i(d\Vphi)=0, \text{ where }\Vphi=(\varphi_1,\varphi_2),$$
which implies
\begin{equation}\label{e2-ex1}
\int_{\C^{\circ}_{i+}}\left(b_{ii}\varphi_i(0)+\hat b_{ii}\varphi_i(-r)\right)\pi_i(d\Vphi)=a_i-\dfrac{\sigma_{ii}}{2}.
\end{equation}
Since $\pi_i$ is an invariant probability measure of $\{\vec X_t\}$, it is easy to see that
\begin{equation}\label{ex1-eq55}
\int_{\C^{\circ}_{i+}}\varphi_i(0)\pi_i(d\Vphi)=\lim_{T\to\infty}\dfrac 1T\int_{0}^T X_{i,t}(0)dt=\lim_{T\to\infty}\dfrac 1T\int_{0}^T X_{i}(t)dt,
\end{equation}
where $(X_{1,t},X_{2,t})=\vec X_t$. Similarly,
\begin{equation}\label{ex1-eq56}
\int_{\C^{\circ}_{i+}}\varphi_i(-r)\pi_i(d\Vphi)=\lim_{T\to\infty}\dfrac 1T\int_{0}^T X_{i}(t-r)dt.
\end{equation}
By virtue of \eqref{ex1-eq55} and \eqref{ex1-eq56}, we can prove that
\begin{equation}\label{ex-eq00}
\int_{\C^{\circ}_{i+}}\varphi_i(0)\pi_i(d\Vphi)=\int_{\C^{\circ}_{i+}}\varphi_i(-r)\pi_i(d\Vphi).
\end{equation}
Combining \eqref{e2-ex1} and \eqref{ex-eq00} yields that
$$
\int_{\C^{\circ}_{i+}}\varphi_i(0)\pi_i(d\Vphi)=\int_{\C^{\circ}_{i+}}\varphi_i(-r)\pi_i(d\Vphi)=\dfrac{a_i-\frac{\sigma_{ii}}2}{b_{ii}+\hat b_{ii}}.
$$
Therefore, we have
$$
\begin{aligned}
\lambda_2(\pi_1)&=\int_{\C^{\circ}_{1+}}\left[a_2-\frac{\sigma_{22}}{2}-b_{21}\varphi_1(0)-\hat b_{21}\varphi_1(-r)\right]\pi_1(d\Vphi)\\
&=a_2-\frac{\sigma_{22}}{2}-\left(a_1-\frac{\sigma_{11}}2\right)\cdot\dfrac{b_{21}+\hat b_{21}}{b_{11}+\hat b_{11}},
\end{aligned}
$$
and
$$
\begin{aligned}
\lambda_1(\pi_2)&=\int_{\C^{\circ}_{2+}}\left[a_1-\frac{\sigma_{11}}{2}-b_{12}\varphi_2(0)-\hat b_{12}\varphi_2(-r)\right]\pi_2(d\Vphi)\\
&=a_1-\frac{\sigma_{11}}{2}-\left(a_2-\frac{\sigma_{22}}2\right)\cdot\dfrac{b_{12}+\hat b_{12}}{b_{22}+\hat b_{22}}.
\end{aligned}
$$
By applying Theorem \ref{thm4.2} to characterize the extinction together with the characterization of  persistence in \cite[Section 4.1]{NNY21}, we have the following results.
\begin{itemize}
	\item If $\lambda_i(\bdelta^*)<0, i=1,2$ then $X_i(t)$ converges to $0$ almost surely with the exponential rate $\lambda_i(\bdelta^*)$ for any initial condition $\Bphi=(\phi_1,\phi_2)\in\C^{\circ}_+$.
	\item If $\lambda_i(\bdelta^*)>0, \lambda_j(\bdelta^*)<0$ for one $i\in\{1,2\}$ and $j\in\{1,2\}\setminus\{i\}$, then $\lambda_j(\pi_i)<0$  and $X_j(t)$ converges to $0$ almost surely with the exponential rate $\lambda_j(\pi_i)$ for any initial condition $\Bphi=(\phi_1,\phi_2)\in\C^{\circ}_+$ and the randomized occupation measure
	converges weakly to $\pi_i$ almost surely.
	\item If $\lambda_i(\bdelta^*)>0$, $i\in\{1,2\}$ and $\lambda_1(\pi_2)<0,\lambda_2(\pi_1)<0$  then $P^{\Bphi}_i>0,i=1,2$ and $P^{\Bphi}_1+P^{\Bphi}_2=1$ where
	$$P^{\Bphi}_i=\PPphi\left\{\U(\omega)=\{\pi_i\}\text{ and }\lim_{t\to\infty}\dfrac{\ln X_j(t)}t=\lambda_j(\pi_i), j\in\{1,2\}\setminus\{i\}\right\}.$$
	\item If $\lambda_1(\bdelta^*)>0, \lambda_2(\bdelta^*)>0$, $\lambda_j(\pi_i)<0, \lambda_i(\pi_j)>0$ for $i,j\in\{1,2\}, i\ne j$ then $X_j(t)$ converges to $0$ almost surely with the exponential rate $\lambda_j(\pi_i)$ and the randomized occupation measure
	converges weakly to $\pi_i$ almost surely for any initial condition $\Bphi=(\phi_1,\phi_2)\in\C^{\circ}_+$.
	\item If $\lambda_1(\bdelta^*)>0, \lambda_2(\bdelta^*)>0$ and $\lambda_1(\pi_2)>0, \lambda_2(\pi_1)>0$,
	any invariant probability measure in $\partial\C_+$ has the form
	$\pi=q_0\bdelta^*+q_1\pi_1+q_2\pi_2$ with $0\leq q_0,q_1,q_2$ and $q_0+q_1+q_2=1$.
	Then, one has
	$\max_{i=1,2}\left\{\lambda_i(\pi)\right\}>0$
	for any $\pi$ having the form above.
	Therefore, there is a unique invariant probability measure $\pi^*$ on $\C^{\circ}_+$.
\end{itemize}
The above characterization generalizes the results of long-term properties in \cite{Liu17}.

Although we only provide the explicit computations for $2$-dimension cases,
our results (in both this paper and \cite{NNY21}) can be applied to characterize the long-time behavior of solutions for stochastic delay Lotka-Volterra competitive models with $n$-species,

\subsection{Stochastic delay Lotka-Volterra predator-prey model}
This section is devoted to the application of our results to stochastic delay Lotka-Volterra predator-prey models.
In contrast to Lotka-Volterra competitive model in which two species compete for food, habitat, territory, etc, the Lotka-Volterra predator-prey models
are
frequently used to describe the dynamics of biological systems in which two species interact, one as a predator and the other one as prey.
The Lotka-Volterra predator-prey system with one prey and two competing predators is given as follows
\begin{equation}\label{eq-ex-2}
\begin{cases}
dX_1(t)=X_1(t)\Big\{a_1-b_{11} X_1(t)-b_{12}X_2(t)-b_{13}X_3(t)\\
\hspace{2cm}-\hat b_{11}X_1(t-r)-\hat b_{12}X_2(t-r)- \hat b_{13} X_3(t-r)
\Big\}dt+X_1(t)dE_1(t),\\[1ex]
dX_2(t)=X_2(t)\Big\{-a_2+b_{21} X_1(t)-b_{22}X_2(t)-b_{23}X_3(t)\\
\hspace{2cm}-\hat b_{21}X_1(t-r)-\hat b_{22}X_2(t-r)- \hat b_{23} X_3(t-r)
\Big\}dt+X_2(t)dE_2(t),\\[1ex]
dX_3(t)=X_3(t)\Big\{-a_3+b_{31} X_1(t)-b_{32}X_2(t)-b_{33}X_3(t)\\
\hspace{2cm}-\hat b_{31}X_1(t-r)-\hat b_{32}X_2(t-r)- \hat b_{33} X_3(t-r)
\Big\}dt+X_3(t)dE_3(t),
\end{cases}
\end{equation}
where
$X_1(t)$, $X_2(t)$, and $X_3(t)$
denote the densities at time $t$ of the prey, and two predators,
respectively;
$a_1>0$
denotes  the growth rate;
$a_2,a_3>0$
represent
the death rate of $X_2,X_3$;
$b_{ii}>0, i=1,2,3$
are the intra-specific competition coefficient of $X_i$;
$b_{ij}\geq 0, i\neq j=1,2,3$, in which $b_{12},b_{13}$ represent the capture rates, $b_{21},b_{31}$ represent the growth from food, and $b_{23}$ and  $b_{32}$ signify the competitions between predators (species 2 and 3);
 $\hat b_{ij}$ is either positive or in $(-b_{ii},0]$;
$r$ is the time delay for each $i,j\in\{1,2,3\}$;
$(E_1(t),E_2(t),E_3(t))^\top=\Gamma^\top\vec B(t)$ with
$\vec B(t)=(B_1(t),B_2(t), B_3(t))^\top$ being a vector of independent standard Brownian motions and
$\Gamma$ being a $3\times 3$ matrix such that
$\Gamma^\top\Gamma=(\sigma_{ij})_{3\times 3}$ is a positive definite matrix.
It is worth noting that system \eqref{eq-ex-2} is the (stochastic delay)
Lotka-Volterra model
with two predators
competing for one prey,
which was considered in \cite{Liu17-dcds}.
If we switch the sign of $a_i$ or $b_{ij},i\neq j$, we can obtain a stochastic time-delay
Lotka-Volterra system with the prey and the mesopredator or intermediate predator.
The case involving a superpredator or top predator, was studied in \cite{Liu17-1,Wu19}, and the stochastic time-delay
Lotka-Volterra system with one predator and two preys was investigated in \cite{Gen17}.

Our assumptions are verified for \eqref{eq-ex-2} in the first part \cite[Section 5.2]{NNY21}.
To characterize the extinction,
first, let us consider the equation on the boundary $\{(0, \varphi_2, \varphi_3)\in\C_+: \varphi_2(s), \varphi_3(s)\geq0\;\forall s\in[-r,0]\}$.
Since $\lambda_i(\bdelta^*)=-a_i-\frac{\sigma_{ii}}{2}<0, i=2,3$,
by applying Theorem \ref{thm4.1} for the space $\{(0, \varphi_2, \varphi_3): \varphi_2(s), \varphi_3(s)\geq0\;\forall s\in[-r,0]\}$, we obtain that there is only one invariant probability measure on $\{(0, \varphi_2, \varphi_3): \varphi_2(s), \varphi_3(s)\geq0\;\forall s\in[-r,0]\}$,
which is $\bdelta^*$.
It indicates that without the prey, both predators die out.

Second, we consider the equation on the boundaries $\C_{12+}:=\{(\varphi_1, \varphi_2, 0)\in\C_+: \varphi_1(s), \varphi_2(s)\geq0\;\forall s\in[-r,0]\}$
and
$\C_{13+}:=\{(\varphi_1, 0, \varphi_3)\in\C_+: \varphi_1(s),\varphi_3(s)\geq0,\;\forall s\in[-r,0]\}$.
If $\lambda_1(\bdelta^*)=a_1-\frac{\sigma_{11}}{2}<0$, an application of Theorem \ref{thm4.1} implies that $\bdelta^*$ is the unique invariant probability measure on $\C_+$.
If $\lambda_1(\bdelta^*)>0$, there is an invariant probability measure $\pi_1$ on $\C^{\circ}_{1+}:=\{(\varphi_1, 0, 0)\in\C_+: \varphi_1(s)>0\;\forall s\in[-r,0]\}$.

In view of Lemma \ref{lm4.1}, we obtain
\begin{equation}\label{e2-ex2}
\int_{\C^{\circ}_{1+}}\left(b_{11}\varphi_1(0)+\hat b_{11}\varphi_1(-r)\right)\pi_1(d\Vphi)=a_1-\dfrac{\sigma_{11}}2.
\end{equation}
Similar to the process of getting \eqref{ex-eq00}, we  obtain from \eqref{e2-ex2} that
$$
\int_{\C^{\circ}_{1+}}\varphi_1(0)\pi_1(d\Vphi)=\int_{\C^{\circ}_{1+}}\varphi_1(-r)\pi_1(d\Vphi)=\dfrac{a_1-\frac{\sigma_{11}}2}{b_{11}+\hat b_{11}}.
$$
Therefore,
$$
\begin{aligned}
\lambda_i(\pi_1)&=\int_{\C^{\circ}_{1+}}\left[-a_i-\frac{\sigma_{ii}}{2}+b_{i1}\varphi_1(0)-\hat b_{i1}\varphi_1(-r)\right]\pi_1(d\Vphi)\\
&=-a_i-\frac{\sigma_{ii}}{2}+\left(a_1-\dfrac{\sigma_{11}}2\right)\cdot\dfrac{b_{i1}-\hat b_{i1}}{b_{11}+\hat b_{11}},\; i=2,3.
\end{aligned}
$$
If $\lambda_1(\bdelta^*)>0$ and $\lambda_i(\pi_1)<0, i=2,3$, in view of Theorem \ref{thm4.1},
there is no invariant probability measure on $\C^\circ_{1i+}.$

By Theorem \ref{thm4.1} and Theorem \ref{thm4.2}, we have the following classification for extinction.
\begin{itemize}
	\item If $\lambda_1(\bdelta^*)<0$ then for any initial condition $\Bphi\in\C^{\circ}_+$, $X_1(t), X_2(t), X_3(t),$ converge to $0$ almost surely with the exponential rates $\lambda_i(\bdelta^*), i=1,2,3$, respectively.
	\item If $\lambda_1(\bdelta^*)>0$, $\lambda_i(\pi_1)<0, i=2,3$ then $X_i(t), i=2,3$ converge to $0$ almost surely with the exponential rate $\lambda_i(\pi_1), i=2,3$, respectively, and the occupation measure converges almost surely for any initial condition $\Bphi\in\C^{\circ}_+$ to $\pi_1$.
	\item If $\lambda_1(\bdelta^*)>0$, $\lambda_i(\pi_1)>0$, $\lambda_j(\pi_{1i})<0$, and $\lambda_j(\pi_{1})<0$  for $i,j\in\{2,3\}$ and $i\ne j$,  then $X_j(t)$ converges to  $0$ almost surely with the exponential rate $\lambda_j(\pi_{1i})$ and the occupation measure converges almost surely for any initial condition $\Bphi\in\C^{\circ}_+$ to $\pi_{1i}$.
	\item If $\lambda_1(\bdelta^*)>0$, $\lambda_2(\pi_1)>0$, $\lambda_3(\pi_1)>0$, $\lambda_j(\pi_{1i})<0$, $\lambda_i(\pi_{1j})>0$ for $i,j\in\{2,3\}$ and $i\ne j$,
	then $X_j(t)$ converges to  $0$ almost surely with the exponential rate $\lambda_j(\pi_{1i})$ and the occupation measure converges almost surely for any initial condition $\Bphi\in\C^{\circ}_+$ to $\pi_{1i}$.
	\item If $\lambda_1(\bdelta^*)>0$, $\lambda_2(\pi_1)>0$, $\lambda_3(\pi_1)>0$, $\lambda_2(\pi_{13})<0$, $\lambda_3(\pi_{12})<0$,  then $p^{\Bphi}_i>0,i=2,3$ and $p^{\Bphi}_2+p^{\Bphi}_3=1$, where
	$$p^{\Bphi}_i=\PPphi\left\{\U(\omega)=\{\pi_{1i}\}\,\text{ and }\,\lim_{t\to\infty}\dfrac{\ln X_i(t)}t=\lambda_i(\pi_{1j}), i\in\{2,3\}\setminus\{j\}\right\}.$$
\end{itemize}
The above assertions
generalize the results in \cite{Liu17-dcds}. Moreover, if we switch the sign of $a_i$ or $b_{ij},i\neq j$
as we mentioned at the beginning of this section and modify slightly
the
above characterization, we improve the results in \cite{Gen17,Liu17-1,Wu19}.

Restricting
our analysis and setting to $\C_{12+}$, which describes the evolution of one predator and its prey, we obtain
\begin{equation}\label{e3-ex2}
\begin{cases}
dX_1(t)=X_1(t)\left\{a_1-b_{11} X_1(t)-\hat b_{11}X_1(t-r)-b_{12}X_2(t)-\hat b_{12}X_2(t-r)
\right\}dt\\
\hspace{2cm}+X_1(t)dE_1(t),\\[1ex]
dX_2(t)=X_2(t)\left\{-a_2+b_{21}X_1(t)+\hat b_{21}X_1(t-r)-b_{22}X_2(t)-\hat b_{22}X_2(t-r)
\right\}dt\\
\hspace{2cm}+X_2(t)dE_2(t).
\end{cases}
\end{equation}
The above characterization can be specialized as:
\begin{itemize}
	\item If $\lambda_1(\bdelta^*)<0$ then $X_1(t), X_2(t)$ converge to 0 almost surely with the exponential rates $\lambda_1(\bdelta^*)$ and
	$\lambda_2(\bdelta^*)$, respectively.
	\item If $\lambda_1(\bdelta^*)>0$ and $\lambda_2(\pi_1)<0$ then
	$X_2(t)$ converges to $0$ almost surely with the exponential rate $\lambda_2(\pi_1)$ and the occupation measure
	converges to $\pi_1$.
\end{itemize}
This result generalizes that of \cite{Liu13}.

\subsection{Stochastic delay replicator equation}
The replicator equation, which is a deterministic monotone, non-linear, and non-innovative game dynamic system plays a
popular and important role in evolutionary game theory.
Such an equation was
introduced in 1978 by Taylor and Jonker in \cite{Tay78}.
Since then significant
contributions have been made in biology \cite{Hof98,Now04},  economics \cite{Wei97},
and
optimization and control for a variety of systems \cite{Bom00, Oba14,Ram10, Tem10}.
 To capture the random factors in nature, the deterministic system has been generalized to
stochastic
systems.
This section is devoted to applying our main results to stochastic delay replicator equation.
The replicator dynamics for a game with $n$ strategies,
involving
social-type time delay (see, e.g., \cite{AM04} for details of such delays)
and white noise perturbation
is given by
\begin{equation}\label{ex5-eq1}
\begin{cases}
dx_i(t)=x_i(t)\bigg(f_i(\vec x(t-r))-\dfrac 1X\displaystyle\sum_{j=1}^n x_j(t)f_j(\vec x(t-r))\bigg)dt\\
\hspace{2cm}+x_i(t)\bigg(\sigma_idB_i(t)-\dfrac 1X\displaystyle\sum_{j=1}^n \sigma_j x_jdB_j(t)\bigg);\;i=1,\dots,n,\\
\vec x(s)=\vec x_0(s);\;t\in[-r,0],
\end{cases}
\end{equation}
where
$X$ is the size of the populations;
$x_i(t)$
is the portion of population that has selected the $i^{th}$ strategy and the distribution of the whole population among the strategy;
the fitness functions $f_i(\cdot):\R^n_+\to\R$, $i=1,\dots,n$ are the payoffs obtained by the individuals playing the $i^{th}$ strategy;
$r$ is the time delay;
and  $\vec x_0(s)\in \Delta_X:=\{\vec x\in\R_+^n:\sum_{i=1}^nx_i=X\}$ for all $s\in[-r,0]$ is the initial value.
Some special cases of \eqref{ex5-eq1}
has been
studied in literature.
For instance,
\cite{Imh09,Imh05}
considered equation \eqref{ex5-eq1} without time delay
in the case $f_i(\cdot):\R^n_+\to \R, i=1,\dots,n$ being linear mappings;
and
\cite{AM04,Oba16} studied
the deterministic version of equation \eqref{ex5-eq1}.

Recall that by a similar argument as in \cite{Oba16,Wei97}, we can show that $\Delta_X$ remains invariant\a.s As a consequence, our assumptions are verified.
Hence, our results in
this paper
(Theorem \ref{thm4.1} and Theorem \ref{thm4.2}) hold for \eqref{ex5-eq1}; see \cite[Section 4.3]{NNY21}.
We first apply our results
to characterize the extinction for some
low-dimensional systems.
Consider equation \eqref{ex5-eq1} for two-dimensional systems.
Define
$$
\C_+^X:=\{(\varphi_1,\varphi_2):\varphi_1(s)+\varphi_2(s)=X\text{ and }\varphi_1(s),\varphi_2(s)\geq 0\text{ for all }s\in[-r,0]\},
$$
$$\partial\C_+^X:=\{(\varphi_1,\varphi_2)\in\C_+^X: \norm{\varphi_1}=0\text{ or }\norm{\varphi_2}=0\},$$
$$\C_{+}^{X,\circ}:=\{(\varphi_1,\varphi_2)\in\C_+^X: \varphi_1(s), \varphi_2(s)>0\text{ for all }s\in[-r,0]\}.$$
In this case, it is clear that there are two invariant probability measures on the boundary $\partial \C_+^X$, which are $\bdelta_1$ and $\bdelta_2$ concentrating on $(X,0)$ and $(0,X)$, respectively, where $0,X$ are understood to be constant functions.
We have
$$\lambda_1(\bdelta_2)=f_1((0,X))-f_2((0,X))-\dfrac{\sigma_1^2+\sigma_2^2}2,$$
$$\lambda_2(\bdelta_1)=f_2((X,0))-f_1((X,0))-\dfrac{\sigma_1^2+\sigma_2^2}2.$$
Using
Theorem \ref{thm4.1}, and Theorem \ref{thm4.2}, we have the following classification for
the extinction
of \eqref{ex5-eq1}:
	If $\lambda_1(\bdelta_2)<0$ (resp., $\lambda_2(\bdelta_1)<0$), there is no invariant probability measure on $\C_+^{X,\circ}$. Moreover, $x_1(t)$ tends to 0 (resp., $x_2(t)$) almost surely.

To proceed, we consider \eqref{ex5-eq1} for three-dimensional systems.
We define the following sets
$$
\begin{aligned} 
\C_+^X:=\Big\{(\varphi_1,\varphi_2,\varphi_3)&:\varphi_1(s)+\varphi_2(s)+\varphi_3(s)=X\text{ and }\\
&\varphi_1(s),\varphi_2(s),\varphi_3(s)\geq 0\text{ for all }s\in[-r,0]\Big\},
\end{aligned}
$$
$$\partial\C_+^X:=\C_{12+}^X\cup\C_{23+}^X\cup\C_{13+}^X,$$
$$
\C_{ij+}^X:=\{(\varphi_1,\varphi_2,\varphi_3)\in\C_+^X: \norm{\varphi_k}=0,k\neq i,j\}, \text{ for }i\neq j\in\{1,2,3\},$$
$$\C_{+}^{X,\circ}:=\{(\varphi_1,\varphi_2,\varphi_3)\in\C_+^X: \varphi_1(s), \varphi_2(s), \varphi_3(s)>0\text{ for all }s\in[-r,0]\}.$$
Denote by $\bdelta_1,\bdelta_2,\bdelta_3$ the invariant probability measures on the boundary $\partial\C_+^X$ of \eqref{ex5-eq1}, concentrating on $(X,0,0)$, $(0,X,0)$, and $(0,0,X)$, respectively. We have
$$\lambda_i(\bdelta_1)=f_i((X,0,0))-f_1((X,0,0))
-\dfrac{\sigma_1^2+\sigma_i^2}2,i=2,3,$$
$$\lambda_j(\bdelta_2)=f_j((0,X,0))-f_2((0,X,0))
-\dfrac{\sigma_2^2+\sigma_j^2}2,j=1,3,$$
and
$$\lambda_k(\bdelta_3)=f_k((0,0,X))-f_3((0,0,X))
-\dfrac{\sigma_3^2+\sigma_k^2}2,k=1,2.$$
If $\max_{j=1,3}\lambda_j(\bdelta_2)>0$ and $\max_{k=1,2}\lambda_k(\bdelta_3)>0$,
there is a unique invariant probability measure on $\C_{23+}^X$, denoted by $\pi_{23}$.
If $\max_{j=1,3}\lambda_j(\bdelta_2)<0$ or $\max_{k=1,2}\lambda_k(\bdelta_3)<0$, the invariant probability measure on $\C_{23+}^X$ does not exist.
If $\pi_{23}$ exists, we have
$$
\begin{aligned} 
\lambda_1(\pi_{23})=-\dfrac{\sigma_1^2}2
+\int_{\C_{23+}^X}\Big(f_1(\Vphi)&-\dfrac{2X\varphi_2(0)
	f_2(\Vphi)+\sigma_2^2\varphi_2^2(0)}{X^2}
\\
&-\dfrac{2X\varphi_3(0)f_3(\Vphi)
	+\sigma_3^2\varphi_3^2(0)}{X^2}\Big)\pi_{23}(d\Vphi).
\end{aligned}
$$
By Lemma \ref{lm4.1} and $\lambda_2(\pi_{23})=\lambda_3(\pi_{23})=0$,
we have
\begin{align*}
\int_{\C_{23+}^X}&\left(\dfrac{2X\varphi_2(0)f_2(\Vphi)+\sigma_2^2\varphi_2^2(0)}{2X^2}+\dfrac{2X\varphi_3(0)f_3(\Vphi)+\sigma_3^2\varphi_3^2(0)}{2X^2}\right)\pi_{23}(d\Vphi)\\
&=\dfrac{\sigma_2^2}2+\int_{\C_{23+}^X}f_2(\Vphi)\pi_{23}(d\Vphi)\\
&=\dfrac{\sigma_3^2}2+\int_{\C_{23+}^X}f_3(\Vphi)\pi_{23}(d\Vphi).
\end{align*}
As a result, one has
\begin{align*}
\lambda_1(\pi_{23})&=-\dfrac{\sigma_1^2+\sigma_2^2}2
+\int_{\C_{23+}^X}\left(f_1(\Vphi)-f_2(\Vphi)\right)\pi_{23}(d\Vphi)\\
&=-\dfrac{\sigma_1^2+\sigma_3^2}2+\int_{\C_{23+}^X}\left(f_1(\Vphi)-f_3(\Vphi)\right)\pi_{23}(d\Vphi).
\end{align*}
The conditions to guarantee the existence of the unique invariant  probability measure $\pi_{12},\pi_{13}$ on the boundary $\C_{12+}^X,\C_{13+}^X$ are similarly obtained and
$\lambda_2(\pi_{13}),\lambda_3(\pi_{12})$ can be
computed similar to $\lambda_1(\pi_{23})$.
Therefore, we have the following classification for
the extinction of
solution of \eqref{ex5-eq1}.
For $l\in\{1,2,3\}$, $X_l(t)$ tends to $0$ almost surely exponentially fast if one of following conditions holds:
\begin{itemize}
	\item $\max_{i\neq l}\lambda_i(\bdelta_l)<0$,
	\item $\max_{i=2,3}\lambda_i(\bdelta_1)>0$, $\max_{j=1,3}\lambda_j(\bdelta_2)>0$, $\max_{k=1,2}\lambda_k(\bdelta_3)>0$ and $\lambda_l(\pi_{ij})<0$, $\{i,j,l\}=\{1,2,3\}.$
\end{itemize}
The explicit characterization of \eqref{ex5-eq1} in $n$-dimensional systems is more complex. However, our results (Theorem \ref{thm4.1}, and Theorem \ref{thm4.2}) as well as that of \cite{NNY21}
are applicable
under
suitable conditions.
Finally, it is worth noting that if $r=0$ (i.e., there is no time delay) and $f_i(\cdot)$, $i=1,\dots,n$ are linear, the characterization of the long-term behavior of \eqref{ex5-eq1} in this section is equivalent to the results in \cite{Imh09,Imh05}.

\subsection{Stochastic delay epidemic SIR model}\label{ex-SIR}
The epidemic SIR model is one of the basic building blocks of compartmental models, from which many
infectious disease
models are derived; and was first introduced by Kermack and
McKendrick in \cite{Ker27,Ker32},
and are deemed effective to depict the spread of many common diseases with permanent immunity such as rubella, whooping cough, measles, and smallpox.
The model consists of three compartments, $S$ (the number of susceptible), $I$ (the number of infectious), and $R$  (the number of recovered (or immune)).
Much
effort
has been devoted to
studying
the behavior of the SIR epidemic systems and its variants; 
see \cite{Die16,DDN19,DN17,DN18} and the references therein. In this subsection, we investigate the stochastic epidemic SIR model with  time delay.
First, we consider the equation with linear incidence rate of the following
form
\begin{equation}\label{ex3-eq1}
\begin{cases}
dS(t)=\left(a-b_1 S(t)-c_1 I(t)S(t)-c_2I(t)S(t-r)\right)dt+S(t)dE_1(t),\\[1ex]
dI(t)=\left(-b_2I(t)+c_1 I(t)S(t)+c_2I(t)S(t-r)\right)dt+I(t)dE_2(t),
\end{cases}
\end{equation}
where
$S(t)$ is the density of susceptible individuals,
$I(t)$ is the density of infected individuals,
$a>0$ is the recruitment rate of the population,
$b_i>0$, $i=1,2$ are the death rates,
$c_i>0$, $i=1,2$ are the incidence rates,
$r$ is the delayed time,
$(E_1(t),E_2(t))^\top=\Gamma^\top\vec B(t)$ with
$\vec B(t)=(B_1(t), B_2(t))^\top$ being a vector of independent standard Brownian motions, and
$\Gamma$ being a $2\times 2$ matrix such that
$\Gamma^\top\Gamma=(\sigma_{ij})_{2\times 2}$ is a positive definite matrix.
It is well-known that
the dynamics of
recovered individuals have no effect on the disease transmission dynamics and that is why we only consider the dynamics of $S(t),I(t)$ in \eqref{ex3-eq1}.

While the conditions for persistence of \eqref{ex3-eq1} were given in \cite[Section 5.4]{NNY21}, we develop the conditions for extinction here.
First, we consider the equation on the boundary $\{(\varphi_1,0):\varphi_1(s)\geq 0\;\forall s\in[-r,0]\}$ and let $\hat S(t)$ be the solution of the equation on this boundary as following
\begin{equation}\label{ex3-eq2}
d\hat S(t)=\left(a-b_1 \hat S(t)\right)dt+\hat S(t)dE_1(t).
\end{equation}
Since the drift coefficient of this equation is negative if $\hat S(t)$ is sufficiently large and positive, if $\hat S(t)$ is sufficiently small,
we can show that there is a unique invariant probability measure $\pi$ of \eqref{ex3-eq1} on $\C_{1+}^\circ:=\{(\varphi_1,0):\varphi_1(s)> 0\;\forall s\in[-r,0]\}$. On the other hand, since $\lambda_2(\bdelta^*)=-b_2-\dfrac{\sigma_{22}}2<0$, there is no invariant probability measure in $\C_{2+}^\circ:=\{(0,\varphi_2):\varphi_2(s)>0;\forall s\in [-r,0]\}$.
We define the following threshold
\begin{equation}\label{ex3-eq3}
\lambda(\pi)=-b_2-\frac {\sigma_{22}}2+\int_{\C_{1+}^\circ}\left(c_1\varphi_1(0)+c_2\varphi_1(-r)\right)\pi(d\Vphi),
\end{equation}
whose sign will be able to characterize the permanence and extinction.
As an application of Lemma \ref{lm4.1}, we get
\begin{equation}\label{ex3-eq4}
\int_{\C_{1+}^\circ}\varphi_1(0)\pi(d\Vphi)=\frac a{b_1}.
\end{equation}
By \eqref{ex-eq00}, we have that
$$
\int_{\C_{1+}^\circ}\varphi_1(-r)\pi(d\Vphi)=\int_{\C_{1+}^\circ}\varphi_1(0)\pi(d\Vphi)=\dfrac{a}{b_1}.
$$
Therefore, under this condition, we obtain from \eqref{ex3-eq3} and \eqref{ex3-eq4} that
\begin{equation*}
\lambda(\pi)=-b_2-\frac {\sigma_{22}}2+\dfrac{a(c_1+c_2)}{b_1}.
\end{equation*}
Using the same idea and technique, it is possible to obtain similar results of Theorem \ref{thm4.1}, and Theorem \ref{thm4.2}.
Therefore, we have the following classifications:
\begin{itemize}
	\item If $\lambda(\pi)<0$, $I(t)$ converges to $0$ almost surely with exponential rate $\lambda(\pi)$ while $S(t)$ tends to $\hat S(t)$.
	\item If $\lambda(\pi)>0$, \eqref{ex3-eq1} has a unique invariant probability measure in $\C_+^\circ$ (follows the first part \cite[Section 5.4]{NNY21}).
\end{itemize}
This characterization is equivalent to the result in \cite{QLiu16-1,QLiu16}.

In the above, we consider the linear incidence to make our computations be more explicit. The characterizations still hold for the following stochastic delay SIR epidemic model with more general incidence rate
\begin{equation}\label{ex3-eq5}
\begin{cases}
dS(t)=\left(a-b_1 S(t)-I(t)f_1\big(S(t),S(t-r),I(t),I(t-r)\big)\right)dt+S(t)dE_1(t),\\[1ex]
dI(t)=\left(-b_2I(t)+I(t)f_2\big(S(t),S(t-r),I(t),I(t-r)\big)\right)dt+I(t)dE_2(t),
\end{cases}
\end{equation}
where
$f_i:\R^4\to \R, i=1,2$ are the incidence functions satisfying
\begin{itemize}
	\item $f_1(0,0,i_1,i_2)=f_2(0,0,i_1,i_2)=0$.
	\item there exists some $\kappa\in (0,\infty)$ such that for all $\Vphi\in\C_+$
	$$
	\begin{aligned} f_2\big(\varphi_1(0),&\varphi_1(-r),\varphi_2(0),\varphi_2(-r)\big)\leq \kappa f_1\big(\varphi_1(0),\varphi_1(-r),\varphi_2(0),\varphi_2(-r)\big)\\
	&\leq\kappa^2\left(1+\abs{\Vphi(0)}+\abs{\Vphi(-r)}\right).
	\end{aligned}
	$$
	\item $f_2(s_1,s_2,i_1,i_2)$ is non-decreasing in $s_1,s_2$ and is non-increasing in $i_1,i_2$.
\end{itemize}
It is important to mention that our conditions are verified by almost all incidence functions used in the literature, including linear functional response, Holling type II functional response, Beddington-DeAngelis functional response, etc.
In the general case, the long-run behavior is almost completely characterized the
same as the case of linear incidence rate by the threshold $\lambda(\pi)$ given by
$$
\lambda(\pi)=-b_2-\frac {\sigma_{22}}2+\int_{\C_{1+}^\circ}f_2\big(\varphi_1(0),\varphi_1(-r),\varphi_2(0),\varphi_2(-r)\big)\pi(d\Vphi),
$$
where
$\Vphi=(\varphi_1,\varphi_2)$ and $\pi$ is the invariant probability measure of \eqref{ex3-eq2}. These results significantly generalize and improve that of
\cite{Che09,QLiu16-2,Mah17}.

\subsection{Stochastic delay chemostat model}\label{ex-che}
Chemostat models,
introduced by Novick and
Szilard in \cite{Nov50},
play an important
role in microbiology, biotechnology, and
population biology.
 A chemostat is a bio-reactor, in which
 fresh medium is continuously added, and culture liquid containing left-over nutrients, metabolic end products, and microorganisms are continuously removed at the same rate to keep a constant culture volume.

This
section studies a model of  $n$-microbial populations competing for a single nutrient in a chemostat, in which we take both
the delayed times
and
the white noises
into consideration. Consider
the following system of stochastic functional differential equations
\begin{equation}\label{ex4-eq1}
\begin{cases}
dS(t)=\bigg(1-S(t)+aS(t-r)-\displaystyle\sum_{i=1}^nx_i(t)p_i(S(t))\bigg)dt+S(t)dE_0(t),\\[2ex]
dx_i(t)=x_i(t)\left(p_i(S(t-r))-1\right)dt+x_i(t)dE_i(t),\;i=1,\dots,n,
\end{cases}
\end{equation}
where
$S(t)$ is the concentration of nutrient at time $t$;
$0\leq a< 1$ is a constant;
$x_i(t),i=1,\dots,n$ are the concentrations of the competing microbial populations;
$p_i(S),i=1,\dots,n$ are the density-dependent uptakes of nutrient by population $x_i$;
$r$ is the delayed time; and
$(E_0(t),\dots,E_n(t))^\top=\Gamma^\top\vec B(t)$ with
$\vec B(t)=(B_0(t),\dots, B_n(t))^\top$ being a vector of independent standard Brownian motions and
$\Gamma$ being a $(n+1)\times (n+1)$ matrix such that
$\Gamma^\top\Gamma=(\sigma_{ij})_{(n+1)\times (n+1)}$ is a positive definite matrix. Moreover,
$\C:=\C([-r,0],\R^{n+1})$ instead of $\C([-r,0],\R^n)$.
While the deterministic version of \eqref{ex4-eq1} was studied with the long-time behavior characterized in \cite{Ell94,Fre89,Wol97},
recent
attention on the stochastic counterpart can be found in
 \cite{Sun18,Sun18-1,Zha19}.

Similar to Section \ref{ex-SIR} as well as \cite[Section 5.5]{NNY21}, if we assume that $p_i:\R\to\R, i=1,\dots,n$ satisfying non-decreasing and bounded properties and $p_i(0)=0$, then our Assumptions hold. Therefore, our results in this paper can be applied to \eqref{ex4-eq1}.
Before
considering the higher dimensional systems,
we consider $n=1$ and $2$.
If $n=1$, there is only one population $x_1$ together with the nutrient $S(t)$.
Similar to Section \ref{ex-SIR}, there is no invariant probability measure of $(S_t,x_{1t})$ in
$\C_{1+}^\circ:=\{(0,\varphi_1)\in\C_+:\varphi_1(s)>0,\forall s\in [-r,0]\},$
where $x_{1t}$ is the memory segment function of $x_1(t)$.
Moreover, there is a unique invariant probability measure $\pi_0$ in $\C_{0+}^\circ:=\{(\varphi_0,0)\in\C_+:\varphi_0(s)>0,\forall s\in [-r,0]\}$.
Hence, it is easy to see that for any invariant probability measure $\pi$ in $\partial\C_+$, we have
\begin{equation*}
\lambda_1(\pi)=\lambda_1(\pi_0)=-1-\dfrac {\sigma_{11}}2+\int_{\C_{0+}^\circ}p_1(\varphi_0(-r))\pi_0(d\Vphi).
\end{equation*}
Therefore, our results
yield the following classification.
\begin{itemize}
	\item If $\lambda_1(\pi_0)>0$ then $(S_t,x_{1t})$ admits a unique invariant probability measure in $\C_+^\circ$; followed by \cite[Section 5.5]{NNY21}.
	\item If $\lambda_1(\pi_0)<0$ then $x_1(t)$ tends to 0 almost surely with exponential rate $\lambda_1(\pi_0)$ while $S(t)$ tends to $\hat S(t)$, where $\hat S(t)$ is the solution of
	$$
	d\hat S(t)=\left(1-\hat S(t)+a\hat S(t-r)\right)dt+\hat S(t)dE_0(t).
	$$
\end{itemize}
To proceed, we study the characterization of the longtime behavior in the case $n=2$.
Similar to the case of $n=1$,
there is no invariant probability measure in
$\C_{i+}^\circ:=\{(0,\varphi_1,\varphi_2)\in\C_+:\norm{\varphi_j}=0, j\neq i \text{ and } \varphi_i(s)>0,\forall s\in [-r,0]\},$
and there is a unique measure $\pi_0$ in
$\C_{0+}^\circ:=\{(\varphi_0,0,0)\in\C_+:\varphi_0(s)>0,\forall s\in [-r,0]\}$.
If $\lambda_i(\pi_0)> 0$, where
\begin{equation*}
\lambda_i(\pi_0)=-1-\dfrac {\sigma_{ii}}2+\int_{\C_{0+}^\circ}p_i(\varphi_0(-r))\pi_0(d\Vphi),\;i=1,2,
\end{equation*}
then there is a unique invariant probability measure $\pi_{0i}$ in
$C_{0i+}^\circ:=\{(\varphi_0,\varphi_1,\varphi_2)\in\C_+:\norm{\varphi_j}=0, j\neq i \text{ and } \varphi_0(s), \varphi_i(s)>0,\forall s\in [-r,0]\}$. Hence, let
\begin{equation*}
\lambda_j(\pi_{0i})=-1-\dfrac {\sigma_{jj}}2+\int_{\C_{0+}^\circ}p_j(\varphi_0(-r))\pi_{0i}(d\Vphi),\;j\neq i.
\end{equation*}
The extinction is classified as follows.
\begin{itemize}
	\item If $\lambda_1(\pi_0)<0$, $\lambda_2(\pi_0)<0$ then $x_1(t),x_2(t)$ tend to 0 almost surely with exponential rate $\lambda_1(\pi_0)$, $\lambda_2(\pi_0)$, respectively, while $S(t)$ tends to $\hat S(t)$, where $\hat S(t)$ is the solution of
	$$
	d\hat S(t)=\left(1-\hat S(t)+a\hat S(t-r)\right)dt+\hat S(t)dE_0(t).
	$$
	\item If $\lambda_i(\pi_0)>0$, $\lambda_j(\pi_0)<0$, $\lambda_j(\pi_{0i})<0$, $i\neq j\in\{1,2\}$ then $x_j(t)$ converges to 0 almost surely with exponential rate $\lambda_j(\pi_{0i})$ and the random occupation measure converges to $\pi_{0i}$.
	\item If $\lambda_1(\pi_0)>0$, $\lambda_2(\pi_0)>0$, $\lambda_i(\pi_{0j})>0$, $\lambda_j(\pi_{0i})<0$, $i\neq j\in\{1,2\}$ then $x_j(t)$ converges to 0 almost surely with exponential rate $\lambda_j(\pi_{0i})$ and the random occupation measure converges to $\pi_{0i}$.
	\item If $\lambda_1(\pi_0)>0$, $\lambda_2(\pi_0)>0$, $\lambda_1(\pi_{02})<0$, $\lambda_2(\pi_{01})<0$ then $q_i>0,i=1,2$ and $q_1+q_2=1$ where
	$$q_i=\PPphi\left\{\U(\omega)=\{\pi_{0i}\}\,\text{ and }\,\lim_{t\to\infty}\dfrac{\ln X_i(t)}t=\lambda_i(\pi_{1j}), i\in\{1,2\}\setminus\{j\}\right\}.$$
\end{itemize}
On the other hand, combining this section and \cite[Section 5.5]{NNY21} leads to that $(S_t,x_{1t},x_{2t})$ admits a unique invariant probability measure in $\C_+^\circ$
if one of following conditions holds
\begin{itemize}
	\item $\lambda_1(\pi_0)>0$, $\lambda_2(\pi_0)<0$, $\lambda_2(\pi_{01})>0$.
	\item  $\lambda_1(\pi_0)<0$, $\lambda_2(\pi_0)>0$, $\lambda_1(\pi_{02})>0$.
	\item  $\lambda_1(\pi_0)>0$, $\lambda_2(\pi_0)>0$, $\lambda_1(\pi_{02})>0$, $\lambda_2(\pi_{01})>0$.
\end{itemize}

For higher dimensional systems,
although, it is somewhat difficult to show concretely in case of general functions $p_i(\cdot)$,
it
is computable in certain
cases.
These classifications improve the results in \cite{Sun18,Zha19}.

\begin{rem}{\rm
		In Sections  \ref{subsec:1}-\ref{ex-che},
		to present
		the main ideas without notation complication
		we used a single delay.
		However,
		the results for models with multi-delays or distributed delays
		can be obtained similarly.
		On the other hand, if $r=0$, i.e., there is no time delay, the above results are consistent with and/or even improve the existing results in the literature.
}\end{rem}

\end{document}